\documentclass[12pt]{article}
\newcommand{\Bin}{\textnormal{Bin}}
\newcommand{\Sp}{S^+}
\newcommand{\Sm}{S^-}

\newcommand{\cN}{c^{(N)}}
\newcommand{\nN}{n^{(N)}}

\newcommand{\pN}{p^{(N)}}

\newcommand{\sN}{\s^{(N)}}
\newcommand{\vsN}{\vs^{(N)}}
\newcommand{\EN}{E^{(N)}}
\newcommand{\cGN}{\cG^{(N)}}
\newcommand{\gN}{g^{(N)}}
\newcommand{\esN}{s^{(N)}}
\newcommand{\lN}{\lambda^{(N)}}
\newcommand{\DlN}{\Dl^{(N)}}
\newcommand{\GN}{G^{(N)}}
\newcommand{\MNf}{M^{(N),f}}
\newcommand{\Tgel}{T_{\mathrm{gel}}}
\newcommand{\ZN}{Z^{(N)}}

\RequirePackage[english]{babel}
\RequirePackage{vmargin}
\RequirePackage{enumitem}
\RequirePackage{amssymb}
\RequirePackage{amsfonts}
\RequirePackage{amsmath}
\RequirePackage{graphicx,epstopdf} % The second one allows to directly use .eps fiugres with pdflatex
\RequirePackage{array}
\RequirePackage[mathcal]{euscript}
\RequirePackage{amsthm}
\RequirePackage{latexsym,amscd}
\RequirePackage{url}
\RequirePackage{dsfont}
\RequirePackage{lmodern} % Removes some OT1 encoding errors (with Beamer in particular) 
\RequirePackage{caption}
\RequirePackage{forloop}
\RequirePackage{hyperref}

\hypersetup{colorlinks,citecolor=green,filecolor=purple,linkcolor=red,urlcolor=blue} 

\setmarginsrb{2cm}{2cm}{2cm}{3cm}{0cm}{0cm}{0cm}{1cm}

\hypersetup{bookmarksdepth = 0}  % Remove navigation panel which opens with Acrobat Reader

\newcommand{\D}{\mathbb{D}}
\newcommand{\E}{\mathbb{E}}

\newcommand{\N}{\mathbb{N}}

\renewcommand{\P}{\mathbb{P}}

\newcommand{\R}{\mathbb{R}}

\newcommand{\cC}{\mathcal{C}}

\newcommand{\cE}{\mathcal{E}}
\newcommand{\cF}{\mathcal{F}}
\newcommand{\cG}{\mathcal{G}}
\newcommand{\cH}{\mathcal{H}}

\newcommand{\cP}{\mathcal{P}}

\newcommand{\cT}{\mathcal{T}}

\renewcommand{\a}{\alpha}
\renewcommand{\b}{\beta}
\newcommand{\g}{\gamma}
\newcommand{\dl}{\delta}
\newcommand{\eps}{\varepsilon}

\renewcommand{\k}{\kappa}
\renewcommand{\l}{\lambda}

\newcommand{\s}{\sigma}
\newcommand{\vs}{\varsigma}

\newcommand{\Dl}{\Delta}

\newcommand{\dfdt}{\frac{\mathrm{d}}{\mathrm{d}t}}

\newcommand{\ds}{\: \mathrm{d}s}
\newcommand{\dt}{\: \mathrm{d}t}

\newcommand{\dds}{\mathrm{d}s}
\newcommand{\ddt}{\mathrm{d}t}

\newcommand{\en}{^{(n)}}
\newcommand{\eN}{^{(N)}}

\newcommand{\la}{\langle}
\newcommand{\ra}{\rangle}
\newcommand{\bsl}{\backslash}
\newcommand{\pinf}{+ \infty}
\newcommand{\eqlaw}{\overset{(d)}{=}}

\newcommand{\un}{\mathds{1}}
\newcommand{\unn}[1]{\mathds{1}_{\left \{ #1 \right \} }}

\newcommand{\lra}{\leftrightarrow}

\renewcommand{\bar}{\overline}

\DeclareMathOperator{\Cov}{Cov}

\newtheorem{thm}{Theorem}[section]
\newtheorem{theorem}[thm]{Theorem}
\newtheorem{prop}[thm]{Proposition}

\newtheorem{lemma}[thm]{Lemma}

\theoremstyle{definition}
\newtheorem{defn}[thm]{Definition}

\theoremstyle{remark}

\makeatletter

\newcounter{i}
\newcounter{aut}
\newcounter{loc}

\renewcommand{\author}[2][1]{\addtocounter{aut}{1} \global \@namedef{type@\arabic{aut}}{#1} \@namedef{author@\arabic{aut}}{#2}}
\newcommand{\location}[2][1]{\addtocounter{loc}{1} \global \@namedef{location@\arabic{loc}}{#2}}
\def \gettype#1{\@nameuse{type@\arabic{#1}}}
\def \getauthor#1{\@nameuse{author@\arabic#1}}
\def \getlocation#1{\@nameuse{location@\arabic#1}}
\def \email{\@namedef{@email}}

\def \maketitle
   {%
   \renewcommand{\thefootnote}{\fnsymbol{footnote}}%
   \begin{center}%
    \LARGE \@title%
    \vskip 0.6cm%
    \forloop{i}{1}{\value{i} < \value{aut}}%
       {%
       \large \getauthor{i}\footnotemark[\gettype{i}] \hspace{1cm} %
       }%
      \large \@nameuse{author@\arabic{aut}}%
       \footnotemark[\@nameuse{type@\arabic{aut}}]%
     \forloop{i}{1}{\value{i} < \value{loc} \or \value{i} = \value{loc}}%
       {%
       \footnotetext[\arabic{i}]{\getlocation{i}}
       }%
     \footnotetext%
       {%
       E-mail address: \@nameuse{@email}%
       }%   
    \end{center}%
	 \renewcommand{\thefootnote}{\arabic{footnote}} \setcounter{footnote}{0}%
  }%

\makeatother

\newcommand{\keywords}{\vspace*{0.3cm} \noindent \textbf{Keywords}: }
\newcommand{\MSC}{\vspace*{0.3cm} \noindent \textbf{MSC 2010}: }

\title{Self-organized criticality in a discrete model for Smoluchowski's equation}
\author[1]{Mathieu \textsc{Merle}}
\author[2]{Raoul \textsc{Normand}}
\location[1]{Laboratoire de Probabilit\'es et Mod\`eles Al\'eatoires, Universit\'e Paris 7, Paris, France}
\location[2]{Institute of Mathematics, Academia Sinica, Taipei, Taiwan}
\email{merle@math.univ-paris-diderot.fr; rnormand@math.sinica.edu.tw}

\begin{document}

\maketitle
\begin{abstract}
We study a discrete model of coagulation, involving a large number $N$ of particles. Pairs of particles are given i.i.d exponential clocks with parameter $1/N$. When a clock rings, a link between the corresponding pair of particles is created {\em only if} its two ends belong to small clusters, i.e. of size less than $\a(N)$, with $1 \ll \a(N) \ll N$. The concentrations of clusters of size $m$ in this model are known to converge as $N \to \infty$ to the solution to Smoluchowski's equation with a multiplicative kernel. Under the additional assumption $N^{2/3} \log^{\gamma}(N) \le \a(N)$, for some $\gamma > 1/3$, we study finer asymptotic properties of this model, namely the combinatorial structure of the graph consisting of small clusters. We prove that this graph is essentially an Erd\H{o}s-R\'enyi random graph, which is subcritical before time 1, and remains critical after time 1. In particular, we show that our model exhibits self-organized criticality at a microscopic level: the limiting distribution of a typical finite cluster is that of a {\em critical} Galton-Watson tree. 
 Our approach allows in particular to verify, under our additional assumption, a conjecture of Aldous. 
\end{abstract}

\keywords Erd\H{o}s-R\'enyi random graph, gelation, hydrodynamic limit, self-organized critica\-lity, Smoluchowski's coagulation equation

\MSC Primary: 60K35; Secondary: 05C80

\section{Introduction}
\subsection{Motivation} 
Smoluchowski's equation, introduced in \cite{Smolu}, is used to describe the coagulation of particles in a mean-field setting. It is characterized by a symmetric kernel $\k(m,m')$ describing at which rate clusters of size $m$ and $m'$ coalesce. The most interesting may be the multiplicative kernel $\k(m,m') = m m'$, where two clusters coalesce at a rate which is proportional to the number of potential links between them. In this setting, Smoluchowski's equation reads
\begin{equation} \label{eq:smolu}
\dfdt c_t(m) = \frac12 \sum_{m'=1}^{m-1} m' (m - m') c_t(m') c_t(m-m') - \sum_{m' \geq 1} m m' c_t(m) c_t(m'),
\end{equation}
with $m \geq 1$, where $c_t(m)$ is meant to represent the concentration of clusters of size $m$. For general references, see \cite{AldousDSM,LaurencotCERM,NorrisCC}. What makes the multiplicative kernel intriguing is that it exhibits \emph{gelation}. Let us explain this phenomenon, which will be fundamental in our analysis. One could perhaps expect that the mass
\[
n_t := \sum_{m \geq 1} m c_t(m)
\]
is preserved over time. At least it seems so by formally differentiating the above term by term and using \eqref{eq:smolu}. However this formal differentiation does not actually hold: it is well-known \cite{EscobedoGMC,FournierWPS,LaurencotCCCFE,LeyvrazSKCP} that in fact, for any non-zero initial conditions, there is a time $\Tgel$ such that the mass is constant on $[0,\Tgel)$ and decreasing afterwards. The usual interpretation of this phenomenon is that one or several clusters of infinite mass are created (and they account for a positive proportion of the total mass). These clusters are called the \emph{gel}. The words \emph{concentrations} and \emph{gel} all come from the chemical interpretation of \eqref{eq:smolu} \cite{Smolu}: one can think of a solution (more precisely a colloid) where clusters of particles interact by coagulating at a multiplicative rate. In the pre-gelation phase, clusters of particles are small enough that they all remain in colloidal suspension. When a cluster becomes too massive, it cannot stay in suspension anymore and goes into a precipitate, or a gel. The time at which the precipitate starts forming is called the gelation time and denoted $\Tgel$. At times past $\Tgel$, clusters of particles which are small enough remain in colloidal suspension, whereas the largest ones form the gel. Of course, more and more particles precipitate as time goes by, which is why the proportion of particles in solution  decreases after $\Tgel$. From now on, we will replace the more accurate {\em in colloidal suspension} with the simpler {\em in solution} to refer to particles which do not belong to the gel.  

Interestingly, some other kernels, such as the constant kernel $\k(m,m') = 1$ or the additive kernel $\k(m,m') = m + m'$ do not exhibit gelation. In those cases,  particles do not coalesce quickly enough to allow infinite clusters to form in finite time, see e.g. \cite{DeaconuSCE} and references therein.

The main issue with gelation is that it makes the system very challenging to study. Phy\-sically, this is due to the presence of two phases in the system. Mathematically, technical issues arise due to the presence of an infinite sum in \eqref{eq:smolu}. This sum is not well-behaved, and does not allow the typical sum-integral inversions: indeed, as we already discussed below \eqref{eq:smolu}, if these inversions were allowed, the mass would remain constant. Consequently, fewer post-gelation results are known. Statements proven in the literature include the following.
\begin{itemize}
\item For well-behaved kernels, such as constant or additive, there is no gelation, and there exists a unique solution to \eqref{eq:smolu}, see \cite{DeaconuSCE,NorrisSCE} and references therein.
\item Existence of gelation for large enough kernels, typically $\k(m,m') \geq (m m')^{\a}$ for $\a > 1/2$ \cite{EscobedoGMC,EscobedoGCFM,LaurencotCCCFE,LeyvrazSKCP}.
\item  For gelling kernels, existence and uniqueness of a solution before gelation \cite{FournierWPS,McLeodISNLDE,NorrisSCE}.
\item For gelling or non-gelling kernels, global existence of solutions through analytic methods \cite{LaurencotGS,LaurencotCCCFE,vanRoesselRCE}.
\item For gelling or non-gelling kernels, convergence (up to a subsequence) of discrete random models to a solution of Smoluchowski's equation \cite{FournierMLP,Jeon,NorrisSCE}, which also gives a probabilistic proof of global existence.
\end{itemize}
The reader will notice that one central point is missing here: the global uniqueness of gelling solutions. As far as we know, the first such post-gelation result appears in \cite{Kokholm}, for Smoluchowski's equation with a multiplicative kernel \eqref{eq:smolu} and initial conditions $c_0(m) = \unn{m=1}$. The extension to general initial conditions was obtained independently in \cite{NZ} and \cite{Rath}. We are not aware of any post-gelation uniqueness results for other kernels.

One fundamental fact about the multiplicative kernel is that it has a very natural probabilistic interpretation. Recall indeed that the rate of coagulation between two clusters corresponds to the number of potential links  between these two clusters. The simplest probabilistic model which exhibits a multiplicative coagulation rate is the following: associate i.i.d. exponential clocks with parameter $1/N$ to pairs of particles, and create a link between particles when the corresponding clock rings. Obviously, two clusters of size $m, m'$ coagulate at a rate proportional to $mm'$ (with a negligible difference for $m \neq m'$). One could perhaps then hope that the concentrations $(\cN(m))$ of clusters of size $m$ in this model converge to a solution of Smoluchowski's equation, but it turns out that this is not the case past time $1$.

In this (pure coagulation) model, the configuration of clusters at time $t$ is exactly an edge percolation on the complete graph with parameter $1-\exp(-t/N) \sim_{N \to \infty} t/N$. 
In other words, the configuration of clusters is exactly that of an Erd\H{o}s-R\'enyi model with $N$ vertices and connection probability $\pN_t:=1-\exp(-t/N) \sim t/N$, which we denote by $ER(N, \pN_t)$. 

As is well-known \cite{ER1, ER2, BollobasRG, JansonRG, vdH}, when $t \le 1$, $ER(N, \pN_t)$  is in its subcritical phase, and in particular, all but a vanishing proportion of the particles belong to finite-size clusters. This shows that $ \sum m c_t(m) =1$ for all $t \le 1$.  When $t>1$ on the other hand, the model is in its supercritical phase, more precisely there exists, with probability tending to one as $N \to \infty$, a giant cluster containing a positive fraction $\zeta(t)$ (which is the survival probability of a Galton-Watson process with Poisson $t$ offspring distribution) of the particles. Also, all but a vanishing proportion of the particles {\em not} in the giant belong to finite-size clusters. This shows that for such a model, any limiting concentration sequence satisfies $\sum m c_t(m)=1-\zeta(t)$ for all $t \ge 0$ (of course $\zeta(t)=0$ for $t \le 1$).  
 
At time $t$, a cluster of mass $m$ links to the giant at a rate proportional to $m \zeta(t)$, and therefore, instead of Smoluchowski's limit we find that for any $m \ge 1$, 
\[
\dfdt c_t(m)= \frac12 \sum_{m'=1}^{m-1} m' (m - m') c_t(m') c_t(m-m') - \sum_{m' \geq 1} m m' c_t(m) c_t(m')-m \zeta(t), 
\]   
and the above is known as \emph{Flory's equation}, see \cite{FournierMLP,NZ, NorrisSCE}. Observe that it coincides with \eqref{eq:smolu} up to time $1$ since $\zeta(t)=0$ at those times, but the additional term $m\zeta(t)$ becomes positive past time $1$.   In fact, though Flory's equation is not as nice looking as Smoluchowski's equation, it is much easier to study, in particular it was earlier known to be well-posed, and many quantities and asymptotics can be studied \cite{EscobedoGMC,NZ,NorrisCC}. 

This is physically reasonable, since as we explained above, Flory's equation comes from a very simple discrete model, which in fact does not distinguish between gel and solution. Indeed, any two clusters, regardless of their sizes, always interact. By contrast, in Smoluchowski's equation, if a positive proportion of particles is in an infinite-size cluster, it should no longer interact with the finite-size clusters. In other words, gel should be inert. Comparing Flory's equation with \eqref{eq:smolu}, it is obvious that the loss of mass happens faster in Flory. Precisely, one can show that the total concentration decreases exponentially fast for Flory's equation, but as $1/t$ for Smoluchowski's, see \cite{NZ}. Even though the above model is not what we are looking for, it will be useful for later comparisons, and we will refer to it as {\em Flory's discrete model}. 

Now, the obvious question is:  how to modify the above microscopic model to instead have convergence to Smoluchowski's equation? 
It should be clear from the previous discussion that one should prevent coalescence of finite-size clusters with large clusters. The idea of \cite{Aldous98TVM, FournierMLP} is to simply introduce a threshold $\a(N)$, which merely needs to verify
\begin{equation} \label{eq:alpha0}
1 \ll \a(N) \ll N
\end{equation} 
below which clusters are considered small, and above which they are considered large and become inert. We will sometimes need the stronger assumption on $(\a(N))$
\begin{equation} \label{eq:alpha}
N^{2/3} \log^{\g} N \ll \a(N) \ll N, \quad \mbox{ for some } \g > 1/3.
\end{equation}
Let us now be more precise with the definition of our model.

\subsection{Definition of the model}  \label{sec:defmodel}
Consider $N \geq 2$ particles $[N] := \{1,\dots,N\}$ and a sequence $\a(N)$ verifying \eqref{eq:alpha0}. On each pair of particles, we attach an independent exponential clock with parameter $1/N$. When a clock rings, we create the corresponding link, \emph{except} if one of its ends belongs to a large\footnote{We use \emph{large} in contrast to \emph{giant}, which is the term usually used in the Random Graphs literature to describe clusters of size of order $N$.} cluster, i.e. a cluster of size $\a(N)$ or greater. In the latter case, this link is not created, and will never be. In other words, the large clusters become inert. From now on we will refer to this algorithm as {\em Smoluchowski's discrete model}, see an example in Figure \ref{fig:example}. At a time $t$, the links whose clock have rung are called \emph{activated} (and thus may or may not be \emph{created}).

\begin{figure}[htb]
\centering
\includegraphics[width=0.75 \columnwidth]{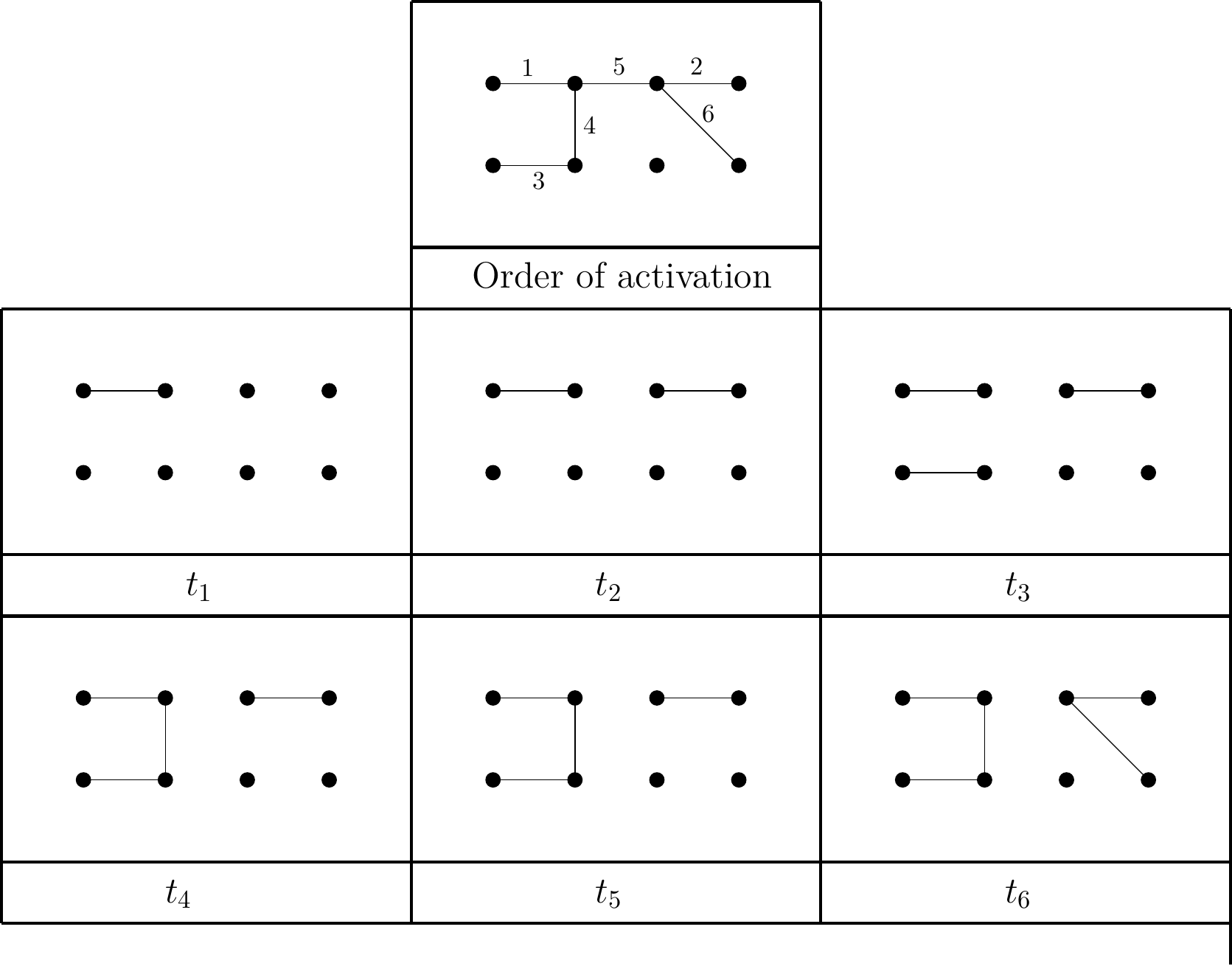}
\caption{An example of Smoluchowski's discrete model, where $N=8$, $\a(N) = 3$ and link $i$ is activated at time $t_i$, with $t_1 < \dots < t_6$. In this case, the system will not evolve after time $t_6$.}
\label{fig:example}
\end{figure}

Clusters of size $\geq \a(N)$ will be called large clusters, and the ones of size $< \a(N)$ called small. We will say that particles in small clusters are \emph{in solution}, whereas particles in large clusters are in the \emph{gel}. When a large cluster is created, we say that it falls into the gel, expressing the fact that it ``precipitates'' and does not interact with particles in solution anymore.

We aim at studying the \emph{configuration} in solution, that is, the graph whose vertices are the particles in solution and whose edges are the \emph{created} links. As is traditional, we will always consider that our processes are c\`adl\`ag.

Let us introduce the following notation.
\begin{itemize}
\item $S(t)$ is the set of particles \emph{in solution} at time $t$.
\item $C(t)$ is the configuration \emph{in solution} at time $t$. Precisely, $C(t) = (S(t), \cE(t))$, where the set of edges is $\cE(t) = \{(i,j) \in S(t)^2 : e_{ij} \le t\}$ and $e_{ij}$ is the clock on the link between $i$ and $j$. In fact, we will always consider graphs (and trees) up to graph isomorphism, but this is a minor technical issue that we will not mention again. 
\item For $S \subset [N]$, $\bar{S} = [N] \bsl S$.
\item $\nN_t = \# S(t) / N$ is the concentration of particles in solution at time $t$.
\item $\pN_t = 1 - e^{-t/N}$ is the probability for a link to be activated at time $t$.
\item $(\tau_i)_{i \geq 1}$ is the sequence of \emph{gelation times}, i.e. successive times when a large component falls into the gel (we adopt the convention $\tau_i=+\infty$ when the total number of falling components ends up smaller than $i$).
\item w.h.p. stands for \emph{with high probability}, and means that a sequence of events has probability tending to 1 as $N \to \infty$.
\item For $S, S' \subset [N]$, a link $S \lra S'$ is a link with an end in $S$ and an end in $S'$.
\item The Poisson distribution with parameter $\l$ is denoted by $\cP(\l)$.
\end{itemize}
The dependency in $N$ of certain quantities such as $\pN$ and $\nN$ is reminded with the exponent $(N)$, but with a slight abuse, we do a bit of bookkeeping and do not write this dependency for quantities such as $S(t)$, $C(t)$ or $\tau_i$.  

\subsection{Results}

Theorem~\ref{th:convconc}, stated below and due to \cite{FournierMLP}, concerns the convergence of concentrations to the solution of Smoluchowski's equation. The result of \cite{FournierMLP} concerns a class of kernels including the multiplicative one. A similar result for multiplicative kernels was proven in \cite{Rath} though the discrete model is slightly different, see Paragraph~\ref{sec:Rath} for more details. 

Denote by $\cN_t(m)$ the concentration of clusters of size $m$ at time $t$, i.e. there are $N \cN_t(m)$ such clusters. As usual (see e.g. \cite{BillingsleyCPM}), $\D(\R^+,E)$ denotes the Skorokhod space of c\`adl\`ag functions on $[0,\pinf)$ with values in a Polish space $E$. We defer to Section \ref{sec:convconc} for more details on Smoluchowski's equation: the precise definition of a {\em solution} is given in Definition \ref{def:solsm}, its existence and uniqueness is stated in Theorem \ref{th:uniqueness}. 

\begin{theorem} \label{th:convconc} 
Assume that $\cN_0 \to c_0$ in $\ell^1$, with $\la c_0, m^2 \ra < \pinf$, and that \eqref{eq:alpha0} holds. Then the process $(\cN)$ converges in distribution in $\D(\R^+,\ell^1)$ to the unique solution of Smoluchowski's equation \eqref{eq:smolu} started from $c_0$. 
\end{theorem}

As mentioned above, Theorem~\ref{th:convconc} was proved in \cite{FournierMLP}, up to two minor differences:
\begin{itemize}
\item the coagulation rates are not exactly the same;
\item the convergence in \cite{FournierMLP} is stated only up to a subsequence, due to the fact that the uniqueness result for \eqref{eq:smolu} of \cite{NZ,Rath} was not yet known.
\end{itemize}

For simplicity, and mainly because Theorems~\ref{th:conv} and \ref{th:typical} below only hold in this case, we chose to describe our discrete model in Paragraph~\ref{sec:defmodel} with $\cN_0(m) = c_0(m) = \unn{m=1}$. However, initial conditions in the discrete model could be modified in an obvious way to have convergence of $\cN_0$ towards any given $c_0$.

For the sake of completeness, and because we feel that it somewhat simplifies and shor\-tens the argument of \cite{FournierMLP} in the multiplicative kernel case, we give in Section \ref{sec:convconc2} a proof of Theorem~\ref{th:convconc}.

The main goal of the paper is instead to study finer properties of our model started from monodisperse initial conditions, in order to be ultimately able to describe clusters in solution. Our main results are the following. 

\begin{theorem} \label{th:conv}
For any $t \geq 0$, conditionally on $\nN_t$, the configuration in solution is that of a
\[
ER \left ( \nN_t N, \pN_t \right )
\]
graph, conditioned on having no large component.

Moreover, if \eqref{eq:alpha} holds, then the sequence $(\nN)$ converges in the Skorokhod space $\D(\R^+,\R)$ to the deterministic function
\begin{equation} \label{eq:nt}
n_t:= \begin{cases}
1 & t \leq 1; \\
1/t & t \geq 1.
\end{cases}
\end{equation}
\end{theorem}

Since, as $N \to \pinf$, $\pN_t \sim t/N$, while $\nN_t \sim n_t$, we observe two distinct regimes. First, when $t<1$, $N \nN_t \times \pN_t \to t <1$, and thus, at such times, we are dealing with a strictly subcritical conditioned ER graph. It is well-known (see e.g. \cite{vdH}, Chapter 4) that the unconditioned graph's largest component is of logarithmic size w.h.p, and therefore the above conditioning is asymptotically trivial. When $t \ge 1$ on the other hand, $N \nN_t \times \pN_t \to 1$, and thus we are dealing with a critical, or at least {\em near-critical} conditioned ER graph, and it seems much harder to deal with the conditioning. Thankfully, in the course of proving Theorem \ref{th:conv} (see Section~\ref{sec:altmodel} and Lemma~\ref{lem:equiv} below), we will couple our model with an alternative model whose marginals are {\em unconditioned} ER graphs. Informally, this means that in the limit as $N \to \infty$, we may simply neglect the effect of the conditioning at all times. In particular, recalling that we consider graphs and trees up to isomorphism, we have the following result.

\begin{theorem} \label{th:typical}
At time $t \ge 0$, pick uniformly at random a particle in solution $i$, and let $\cC_{\mathrm{typ}}(t)$ denote the graph rooted at $i$ consisting of the connected component of $i$. Then, for any rooted finite tree $\cT$, 
\[
\lim_{N \to \infty} \P( \cC_{\mathrm{typ}}(t) = \cT) = \P(\mathrm{GW}(\cP(t n_t)) = \cT),
\]
where $\mathrm{GW}(\cP(\l))$ denotes a Galton-Watson tree with offspring distribution $\cP(\l)$. 

In particular, at any time $t \ge 1$, the limiting distribution of a typical cluster in solution is that of a {\em critical} Galton-Watson tree with $\cP(1)$ offspring distribution. 
\end{theorem} 

Theorem \ref{th:typical} shows that Smoluchowski's discrete model exhibits self-organized criticality (SOC), and that criticality is exactly reached at the gelation time $1$. For a more precise description, see Sections \ref{sec:rel} and \ref{sec:conclusion}. 

Moreover, Theorems \ref{th:conv} and \ref{th:typical} show part of Aldous' conjecture 3.6 of \cite{Aldous98TVM}. In fact, most of the conjecture turns out to be true, as we explain in Sections \ref{sec:introAldous} and \ref{sec:Aldousconj}. 

\subsection{Organization of the paper}

After recalling some results on Smoluchowski's equation, we prove Theorem \ref{th:convconc} in Section \ref{sec:convconc2}. The proofs of Theorems \ref{th:conv},\ref{th:typical} take up  the remainder of the paper. It will be organized in the following main steps.
\begin{itemize}
\item In Section \ref{sec:ER}, we prove some precise estimates on the largest component of a near-critical (and slightly above the critical window) Erd\H{o}s-R\'enyi graph (Theorem \ref{th:sizecomp}). Later in that section, we see how these translate into information about the first gelation event (Proposition \ref{prop:atgel}). 
\item In Section \ref{sec:struct}, we describe the combinatorial structure of particles in solution, which yields a proof for the first part of Theorem \ref{th:conv} (Lemmas \ref{lem:struct1} and \ref{lem:struct2}).
\item Section \ref{sec:altmodel} introduces an alternative model constructed from a family of Erd\H{o}s-R\'enyi graphs. We prove that it has, with high probability, the same distribution as our original model (Lemma \ref{lem:equivasymp}). This allows us to give precise estimates on the gelation times (Proposition \ref{prop:geltimes}). 
\item These estimates are used in Section \ref{sec:conv} to prove the tightness of $(\nN)$, and show that it has a unique limit point given by \eqref{eq:nt}.   
\end{itemize}

Finally, in Section \ref{sec:conclusion}, we explain how Theorem \ref{th:typical} follows from Theorem \ref{th:conv}, along with other concluding comments.

\subsection{Related works} \label{sec:rel}

\subsubsection{Known results on Smoluchowski's equation}

As we already mentioned above, Theorem~\ref{th:convconc}, up to minor differences, was proven in \cite{FournierMLP}. In this paper, $(\cN)$ is considered as a pure-jump Markov process, and its dynamics is given by the jump rates. Note however that the model we consider contains much more information about the system: for instance, the shape of the clusters (in terms of their graph structure) is lost in the mere knowledge of $(\cN)$. 
In some sense, keeping track of the graph structure allows to look at phenomena on a microscopic scale.

 As was also mentioned, \cite{NZ} and \cite{Rath} independently obtained the uniqueness of the solutions to Smoluchowski's equation, along with various properties of the solutions. Both papers provide expressions for the moment generating functions of the concentrations, which allow to recover all these properties.

\subsubsection{Frozen percolation on the binary tree}

In \cite{AldousFP}, Aldous considers a model of frozen percolation on a binary tree. Edges are given independent uniform $[0,1]$ clocks, and when a clock rings, the edge appears if and only if both vertices belong to finite clusters. A rewording of the latter condition is that any cluster which becomes of infinite size becomes instantaneously frozen and can no longer interacts with other clusters. 

Gelation happens when the first infinite cluster is formed, at time $1/2$. After proving existence of such a process past time $1/2$, Aldous computes in particular the rate at which the cluster containing a given edge or a given vertex becomes frozen. He further shows that at any time past gelation, conditionnally given that an edge $e$ is present and its cluster is finite, the cluster of $e$ is a {\em critical} Galton-Watson tree with two ancestors. On the other hand, the distribution of infinite clusters is always that of an incipient infinite cluster (in this case a Galton-Watson with $\Bin(2,1/2)$ offspring distribution conditioned to be infinite). Thus, \cite{AldousFP} is amongst the first papers of the mathematical literature where the phenomenon of SOC is rigourously established. In Section 5 of \cite{AldousFP}, Aldous further suggests looking at several related models which should exhibit similar properties. 

\subsubsection{Aldous' conjectures on the present model} \label{sec:introAldous}

In Sections 3.7 of \cite{Aldous98TVM} and 5.5 of \cite{AldousFP}, Aldous defines the model of the current article. He makes several conjectures, under the assumption \eqref{eq:alpha0}. Not surprisingly, we will only be able to discuss the validity of these conjectures under the more restrictive assumption \eqref{eq:alpha}. 

The main assessment in Conjecture 3.6 of \cite{Aldous98TVM} is a process version of Theorem~\ref{th:typical}. More precisely, Aldous defines in Proposition 3.5 of \cite{Aldous98TVM} a tree-valued, non-homogeneous Markov process related to Galton-Watson trees with Poisson offspring distribution, which he denotes  $(\cG^0(s), s \in [0,1/V])$. It starts with a single vertex, and between times $s$ and $s+\dds$, an independent $GW(\cP(s n_s))$ is added with probability $n_s \ds$ independently at each vertex in the tree. We have denoted $1/V = \inf\{s : |\cG^0(s)| = \infty\}$, as it turns out that $V \sim \mathrm{Unif}[0,1]$, so that $1/V \sim \mathrm{Pareto}(1,1)$.  

Let $(\cC(t))$ be the component containing the vertex 1 at time $t$, and $\ZN$ the time it ends up in the gel. Aldous conjectures (beware of the missing exponent $0$ in the statement of Conjecture 3.6) that, in the sense of \emph {local weak convergence}, one has  
\begin{equation} \label{eq:aldousconj}
(\cC(t), t \le \ZN) \underset{N \to \infty}{\longrightarrow} (\cG^0(t), t \le 1/V).
\end{equation}
Let us briefly explain why Theorems~\ref{th:conv} and \ref{th:typical} ensure the above.

The fact that the distribution of $\ZN$ converges to that of a Pareto$(1,1)$ follows directly from the exchangeability of particles and Theorem~\ref{th:conv}. Moreover, before $1$ falls into the gel (i.e., conditionally on $t < \ZN$), between times $t$ and $t+\ddt$, a new edge from any given vertex of the cluster $\cC(t)$ is created with probability $\nN_t \dt$. It links this vertex to the cluster of a uniformly chosen vertex in solution. 
By Theorems~\ref{th:conv} and \ref{th:typical}, when $N \to \infty$, $\nN_t \to n_t$, and the distribution of the added cluster converges to that of a $GW(\cP(t n_t))$ tree. This ensures that the asymptotics of the rate and distribution of the cluster added at a given vertex of $\cC(t)$ are exactly those of the cluster added at a given vertex of $\cG^0(t)$. 

This suffices to show Conjecture 3.6 of \cite{Aldous98TVM}: we only want local weak convergence, and therefore we only need to control what happens in balls of given radius centered at $1$, what always involve finitely many vertices. 

Note that a stronger type of convergence up to the explosion time would fail. For instance, there are many cycles in $\cC(Z)$, so its distribution must differ from that of $\cG^0(1/V)$, which is a tree (precisely, a $GW(\cP(1))$ conditioned on being infinite). This does not put weak local convergence in jeopardy, because the minimal length of these cycles typically diverges with $N$.

 In Section \ref{sec:Aldousconj}, we will discuss some additional features of this conjecture.

\subsubsection{SOC in a model of frozen percolation} \label{sec:Rath}

 Self-organized-criticality for a very closely related model of frozen percolation was obtained by R\'ath in \cite{Rath}. In this work, edges appear at rate $1/N$, while particles are stricken at a fixed rate $\a(N)/N$ (with $(\a(N))$ satisfying \eqref{eq:alpha0} for the most interesting behavior). The effect of a particle being stricken is that its connected component instantaneously becomes frozen, and no longer interacts with particles which remain in solution.

 R\'ath shows that there is convergence of the concentration of clusters of size $m$ towards the solution of Smoluchowski's equation with a multiplicative kernel. He obtains this result for both monodisperse and polydisperse initial conditions, which corresponds, once again, to Theorem~\ref{th:convconc} of the present paper, and was obtained independently from \cite{Rath} in \cite{FournierMLP}. He further investigates properties of the limiting concentrations. In particular, the characterization of SOC is through the fact that at any time past gelation, the tail of the mass satisfies $\sum_{m \ge k} m c_t(m) \sim C(t) k^{-1/2}$, with $C(t)>0$, in contrast to an exponential decay before gelation. Note however that in the case (which is the only one we consider in the present paper, except for Theorem~\ref{th:convconc}) of monodisperse initial conditions, this property follows very easily from the earlier knowledge of an explicit solution to Smoluchowski's equation from \cite{Kokholm}, see Section \ref{sec:macro} for more details. For polydisperse initial conditions, this could also be obtained from the formulas of \cite{NZ}, using Tauberian theorems.

The approach in the current paper is to try to go beyond the description in terms of concentrations of clusters of a given size, by obtaining a description of the graph structure of particles in solution. Theorem \ref{th:conv} and our coupling of Section~\ref{sec:altmodel} ensure in particular that this graph structure is very close to that of a ER model, asymptotically subcritical before gelation and critical afterwards. 

Moreover, Theorem \ref{th:typical} is a more precise and perhaps more striking way of exhibiting SOC: it asserts that the limiting distribution of a typical cluster is that of a finite, critical tree (with offspring distribution $\cP(1)$). In particular, as we detail in Section~\ref{sec:macro}, the characterization of SOC in the sense of \cite{Rath,RathToth} can be easily recovered through Theorem~\ref{th:typical}: the probability for a critical Galton-Watson tree to be of size greater than $k$ decays as $k^{-1/2}$.  

Although our model is slightly different from that of \cite{Rath}, we believe that starting from monodisperse initial conditions, a minor adaptation of our methods would yield results similar to Theorem \ref{th:conv} and Theorem \ref{th:typical}. In particular, we conjecture that past gelation, typical clusters are again $\cP(1)$ Galton-Watson trees.

\subsubsection{SOC in a forest-fire model}  

Self-organized-criticality was also obtained in the forest-fire model of R\'ath and T\'oth \cite{RathToth}. Here, edges appear or re-appear at rate $1/N$, while particles are stricken at rate $\a(N)/N$, with $(\a(N))$ satisfying \eqref{eq:alpha0}. In this case however, the effect of a particle being stricken is that edges in the corresponding connected component simply vanish, but vertices remain.

This model also presents obvious similarities to ours, but behaves in fact somewhat differently. Indeed, since only the edges are affected by the coagulation-fragmentation dynamics, the total mass remains constant. This is why, in the limit past gelation, instead of Smoluchowski's equation, the authors obtain a system of {\em constrained} ODE's. The constraint simply is that the total mass remains $1$. The system of ODE's is closely related to Smoluchowski's equation: on the one hand the equation remains unchanged for the evolution of concentrations of clusters of mass $m \ge 2$. On the other hand, the evolution of clusters of mass 1 is modified, as they not only disappear due to coagulation with other clusters, but also appear due to the effect of lightning (in such a way that the total mass remains constant). Importantly, a solution to the constrained system of ODE's found in \cite{RathToth} is, past time $1$, very different from a solution to Smoluchowski's equation: indeed, the presence of a greater number of isolated vertices also affects the growth of larger clusters. R\'ath and T\'oth establish SOC for this model, and the characterization of SOC is the same as in \cite{Rath}, through the decay of the tail of the mass.

Concerning the shape of the clusters, we conjecture that in this case, typical clusters are again subcritical Galton-Watson trees before gelation, and critical Galton-Watson trees past gelation.  Note however that the precise offspring distribution past gelation has to be more complicated than in our case and that of \cite{Rath}, due to the a priori changing proportion of isolated vertices. In particular, it can no longer be stationary.

\subsubsection{SOC in a discrete model of limited aggregation}
   
In a forthcoming paper \cite{MN}, we study a similar model, but where particles initially have a certain number of arms, used to perform coagulations. This model of coagulation with limited aggregation was introduced in the continuous setting by Bertoin and studied in \cite{BertoinTSS}, before gelation. 

The corresponding discrete model is to pair arms uniformly at random at some rate, but only when they both belong to small clusters. Some of the techniques of the present paper can be applied, and others have to be introduced. We manage to prove again self-organized criticality: past the gelation time, a typical cluster in solution is a delayed critical Galton-Watson tree, and the distribution of the reproduction law can be given. 

The main differences between the present model and that of \cite{MN} are the following.
\begin{itemize}
\item The convergence of concentrations is towards a solution of Smoluchowski's equation with {\em limited aggregations} (see \cite{NZ}). 

\item The model of the present paper at a given time is related to a (conditioned) Erd\H{o}s-R\'enyi random graph, while the model in \cite{MN} is related to a (conditioned) {\em configuration model}. 

\item An important parameter of the model in \cite{MN} is the initial distribution of arms. The behavior of the model depends heavily on it: for instance, gelation occurs only when there are sufficiently many arms to begin with. Also, for most initial distribution of arms (except the Poisson ones), the distribution of a typical cluster is no longer stationary past gelation.   
\end{itemize}
For the readers interested in \cite{MN}, the present paper should be an appropriate and more accessible introduction.

\section{Convergence to Smoluchowski's equation} \label{sec:convconc}

\subsection{Definition and well-posedness of Smoluchowski's equation}
  
Smoluchowski's equation with a multiplicative kernel is given by \eqref{eq:smolu}. In this paragraph, we recall results about this equation, and fix some inaccuracies from \cite{NZ}. To begin with, we need to define what we mean by a solution. Let, for $f, g : \N \to \R^+$,
\[
\la f, g \ra = \sum_{m \geq 1} f(m) g(m).
\]
With a slight abuse of notation, we will, in this paragraph, write $m$ for the function $f(m) = m$, $m^2$ for the function $f(m)=m^2$, and so on.
\begin{defn} \label{def:solsm}
We call a family $(c_t(m), m \geq 1)_{t \geq 0}$ of nonnegative continuous functions a solution to Smoluchowski's equation with initial conditions $c_0 \in [0,\pinf)^{\N}$ if
\begin{itemize}
\item for every $t \geq 0$, $\int_0^t \la m, c_s \ra^2 \ds < \pinf$;
\item for every $t \geq 0$ and $f : \N \to [0,\pinf)$ with compact support,
\begin{equation} \label{eq:smolu2}
\la c_t, f \ra - \la c_0, f \ra = \frac12 \int_0^t \sum_{m, m'=1}^{\pinf} m m' c_s(m) c_s(m') \left( f(m+m') - f(m) - f(m') \right )\ds.
\end{equation} 
\end{itemize}
\end{defn}

Note that the RHS of \eqref{eq:smolu2} is well-defined by the first part of the definition, and, with $f = \unn{m}$, it is just the integrated form of \eqref{eq:smolu}. Extending to $f$ with bounded support is simply the linearity of the equation. Unlike the definition given in \cite{NZ}, we do not assume that $\la c_t, m^2 \ra$ is bounded in a neighborhood of 0 whenever it holds at time 0. It turns out that this was an unnecessary assumption, as we shall now prove.

\begin{lemma} 
Assume that $(c_t)$ is a solution to \eqref{eq:smolu} with $\la c_0, m \ra < \pinf$. Then $\la c_t, m \ra \leq \la c_0, m \ra$ for all $t \geq 0$. In particular, \eqref{eq:smolu2} extends to all bounded $f : \N \to [0,\pinf)$.
\end{lemma}

\begin{proof}
The first part of the above is exactly Lemma 2.4 in \cite{NZ}. Now, fix $f$ as in the statement, take $t > 0$ and write $f^b(m) = f(m) \unn{m \leq b}$. Consider \eqref{eq:smolu2} with $f^b$. By monotone convergence,
\[
\lim_{b \to \pinf} \la c_0, f^b \ra = \la c_0, f \ra , \quad \lim_{b \to \pinf} \la c_t, f^b \ra = \la c_t, f \ra.
\]
Moreover,
\begin{align*}
\left | \sum_{m,m' \geq 1} m m' c_s(m) c_s(m') \left ( f^b(m+m') - f^b(m) - f^b(m') \right ) \right | & \leq 3 \sup_{m \in \N} f(m) \times \la m, c_s \ra^2 \\
& \leq 3  \sup_{m \in \N} f(m) \times \la m, c_0 \ra^2 ,
\end{align*}
and we may then pass to the limit in \eqref{eq:smolu2} as $b \to \pinf$ by bounded convergence.
\end{proof}

The following result shows that our definition of a solution and the one given in \cite{NZ} coincide.

\begin{lemma}
Assume that $(c_t)$ is a solution to \eqref{eq:smolu} with $\la c_0, m^2 \ra < \pinf$. Then $\la c_t, m^2 \ra$ is bounded in a neighborhood of 0.
\end{lemma}

\begin{proof}
From the previous result, we may take, for $b \geq 0$, $f^b(m) = (m \wedge b)^2$ in \eqref{eq:smolu2}. One readily checks that
\[
f^b(m+m') - f^b(m) - f^b(m') \leq 2 m m' \unn{m \leq b} \unn{m' \leq b}.
\]
Plugging this in \eqref{eq:smolu2} shows that
\[
\la c_t, f^b \ra - \la c_0, f^b \ra \leq \int_0^t \left ( \sum_{m=1}^b m^2 c_s(m) \right ) \ds \leq \int_0^t \la c_s, f^b \ra^2 \ds.
\]
This is an integral inequality \`a la Gronwall of the type $u_t \leq a + \int_0^t u_s^2 \ds$. If one denotes by $v$ the RHS, then $v' = u^2 \leq v^2$, whence one deduces by integrating that $v_t \leq (1/a - t)^{-1}$, for $t \leq 1/a$, and the same inequality holds for $u$ since $u \leq v$. Hence, the previous inequality implies that
\[
\la c_t, f^b \ra \leq \frac{1}{1/\la c_0, f^b \ra - 1}, \quad t \leq 1/\la c_0, f^b \ra.
\]
By monotone convergence, we thus get
\[
\la c_t, m^2 \ra \leq \frac{1}{1/\la c_0, m^2 \ra - 1}, \quad t \leq 1/\la c_0, m^2 \ra,
\]
which proves the result.
\end{proof}

This being done, we may recall the well-posedness result obtained in Theorem 2.2 and Proposition 2.6 of \cite{NZ}. For convenience, we rephrase it in our context. 

\begin{theorem} \label{th:uniqueness}
Consider initial conditions $(c_0(m))$ with
\[
m_0 = \la m, c_0 \ra \in (0,\pinf], \quad K = \la m^2, c_0 \ra \in (0,\pinf],
\]
and define
\[
g_0(x) = \la m x^m, c_0 \ra, \quad x \in [0,1].
\]
Let
\[
\Tgel = 1 / K \in [0,\pinf).
\]
Then the following hold.
\begin{enumerate}
\item Smoluchowski's equation \eqref{eq:smolu} has a unique solution defined on $\R^+$.
\item For $t > \Tgel$, the equation
\[
\ell_t g(\ell_t) = \frac1t
\]
has a unique solution $\ell_t \in (0,1)$. The mass in solution $n_t = \la m, c_t \ra$ is given by
\[
n_t = 
\begin{cases}
1, & t \leq \Tgel, \\
g_0(\ell_t), & t > \Tgel.
\end{cases}
\]
\item In particular $(n_t)$ is continuous, constant on $[0,\Tgel]$, strictly decreasing on $[\Tgel,\pinf)$.
\end{enumerate}
\end{theorem}

In other words, \eqref{eq:smolu} is well-posed, and by the third point above, some mass starts to disappear at time $\Tgel$, i.e. there is gelation at $\Tgel$.

\subsection{Convergence} \label{sec:convconc2}

Recall that we denote by $\cN_t(m)$ the concentration of clusters of size $m$ at time $t$, i.e. there are $N \cN_t(m)$ such clusters at time $t$. Let us call a link between particles \emph{relevant} (at time $t$) if
\begin{itemize}
\item it is not activated;
\item both its ends are in small clusters;
\item its ends are in two different clusters.
\end{itemize}
The point of that notion is that a coagulation occurs if and only if a relevant link is activated. How many such links are there? Consider a small cluster of size $m$. Each particle in this cluster has $\cN_t N - m$ relevant links, namely the links with particles in solution not belonging to the same cluster. Hence, there are $m (\cN_t N - m)$ relevant links starting from this cluster. When we sum these quantities over all clusters in solution, each relevant link ends up being counted twice, we therefore obtain a total of
\[
\frac12 \sum_{m = 1}^{\a(N)-1} N \cN_t(m) m (\cN_t N - m)
\]
relevant links at time $t$. Each is activated at rate $1/N$, so the next coagulation event happens at rate
\[
\lN(\cN_t) =
\frac12 \sum_{m = 1}^{\a(N)-1} \cN_t(m) m (\cN_t N - m).
\]
Only these events make $(\cN_t(m))$ change, which should make it clear that $(\cN_t(m), m \geq 1)_{t \geq 0}$ is a pure-jump Markov process, with values in $\ell^1$, say. By the same reasoning, it is easy to compute the rates. For $m \neq m'$, there are
\[
m N \cN_t(m) m' N \cN_t(m')
\]
relevant links between small clusters of size $m$ and $m'$, and
\[
\frac12 (m N \cN_t(m) - m) m N \cN_t(m)
\]
relevant links between small clusters of the same size $m$. Assume then that the system is in a state $\eta \in \ell^1$. Define $\unn{m} \in \ell^1$ to consist only of zeroes, except for $1$ at $m$. Let
\[
\DlN(m,m') = \frac1N \left ( \unn{m+m'} - \unn{m} - \unn{m'} \right ).
\]
Then, if the system is in a state $\eta \in \ell^1$, it jumps to $\eta + \DlN(m,m')$ at rate
\[
\lN_{m,m'}(\eta) =
\begin{cases}
m m' \eta(m) \eta(m') N & \text{if $m \neq m'$} \\
m m \eta(m) (\eta(m) - 1/N) N / 2 & \text{if $m = m'$.}
\end{cases}
\]
Hence, the process $(\cN)$ is a pure-jump Markov process on $\ell^1$ with generator 
\begin{equation} \label{eq:gene}
\GN \mathcal{F}(\eta) = \frac12 \sum_{m,m'=1}^{\a(N)-1} \left ( \mathcal{F} \left (\eta + \DlN(m,m') \right ) - \mathcal{F}(\eta) \right ) \lN_{(m,m')}(\eta),
\end{equation}
with $\mathcal{F} : \ell^1 \to \mathbb{R}$. In order to prove Theorem \ref{th:convconc}, we classically proceed in two steps, first proving tightness, then the uniqueness of the limit points. The technical results all appear in \cite{EthierKurtz}, see also \cite{BillingsleyCPM}.

\begin{lemma} \label{lem:tightness}
Assume that $\cN_0 \to c_0$ in $\ell^1$, with $\la c_0, m \ra < \pinf$. Then the process $(\cN)$ is tight in $\D(\R^+,\ell^1)$, and any limit point is a continuous function from $\R^+$ to $\ell^1$.
\end{lemma}

\begin{proof}
For the tightness, we can use Aldous' criterion. The norm we consider is the $\ell^1$-norm $| \cdot |$. There are two points to prove. First, we need to establish that
\[
\lim_{a \to \pinf} \limsup_{N \to \pinf} \P \left ( \sup_{t \leq T} |\cN_t| > a \right ) \to 0
\]
for all $T \geq 0$, which is obvious since $|\cN_t| \leq |\cN_0| \to |c_0| < \pinf$. Secondly, we need to check that for $\eps, \eta , T > 0$, there is a $\dl_0$ and $N_0$ such that for any $\dl < \dl_0$, $N \geq N_0$ and $\tau$ a discrete $\cN$-stopping time with $\tau \leq T$, we have
\[
\P \left ( |\cN_{\tau + \dl} - \cN_{\tau}| \geq \eps \right ) \leq \eta.
\]
But any jump of $\cN$ has size bounded (in $\ell^1$) by $3/N$. Moreover, at any $\cN$ stopping time $\s$, the next jump occurs after an exponential time with rate $\lN(\cN_{\s})$. Now
\[
\lN(\cN_{\s}) \leq \frac12 \cN_{\s} \sum_{m = 1}^{\a(N)-1} \cN_{\s}(m) m  N \leq \frac12 N \cN_{\s} \sum_{m = 1}^{\a(N)-1} \cN_{0}(m) m \leq K N
\]
for some constant $K$ depending only on $(\cN_0)$. The expectation of number of jumps on the interval $[\tau,\tau+\dl]$ is thus bounded by $K N \dl$, so the result just follows from Markov's inequality for small enough $\dl$. The continuity of the limit points simply follows from the fact that the size of the jumps, at most $3/N$, tends to 0.
\end{proof}

We then prove that any limit point of this sequence solves Smoluchowski's equation. Up to a subsequence, and using Skorokhod's representation theorem, we may assume that $\cN \to c$ a.s. in $\D(\R^+,\ell^1)$. Since $c$ is continuous, there is even uniform convergence on compact sets. 

We now fix any $T > 0$, and prove below that $c$ is a.s. a solution to \eqref{eq:smolu} on $[0,T]$. This is of course enough to ensure that $c$ is a.s. a solution to \eqref{eq:smolu}, in the sense of Definition \ref{def:solsm}. Since \eqref{eq:smolu} has a unique solution by Theorem \ref{th:uniqueness}, it will suffice to ensure that $(\cN)$ does converge to this solution. 

To begin with, for verifying \eqref{eq:smolu2}, it suffices to take $\cF$ linear in \eqref{eq:gene}: for such $\cF$, $\cF(\eta) = \sum_m \eta(m) \cF(\unn{m})$, so letting $f:\mathbb{N} \to \mathbb{R}$ such that $f(m)=\cF(\unn{m})$ we get $\cF(\eta) = \la f, \eta \ra$. With a slight abuse of notation we will now confuse $\cF$ and $f$, and we may write 
\begin{equation} \label{eq:GN_f}
\begin{split}
\GN f(\eta) ={} & \frac12 \sum_{m,m'=1}^{\a(N)-1} \left ( f(m+m') - f(m) - f(m') \right ) m m' \eta(m) \eta(m') \\
& - \frac1N \sum_{m=1}^{\a(N)-1} \left ( f(2m) - 2 f(m) \right ) m^2 \eta(m).
\end{split}
\end{equation}
Recall as well that
\begin{equation} \label{eq:mart}
\MNf_t := f(\cN_t) - f(\cN_0) - \int_0^t \GN f(\cN_s) \ds
\end{equation}
is a martingale for any bounded $f$. 

Take $f^b(m) = m \wedge b$ for some $b > 0$ in \eqref{eq:GN_f}. It is easy to check that
\[
- b \leq f^b(m+m') - f^b(m) - f^b(m') \leq - b \unn{m \geq b} \unn{m' \geq b},
\]
so that
\begin{align*}
\GN f^b(\cN_s) & \leq - \frac{b}{2} \left (\sum_{m = b}^{\a(N)-1} m \cN_s(m) \right )^2 + b \frac{\a(N)}{N} \sum_{m=1}^{\a(N)-1} m \cN_s(m) \\
& \leq - \frac{b}{2} \left (\sum_{m = b}^{\a(N)-1} m \cN_s(m) \right )^2 + b \frac{\a(N)}{N} \sum_{m=1}^{\a(N)-1} m \cN_0(m) \\
& \leq - \frac{b}{2} \left (\sum_{m = b}^{\a(N)-1} m \cN_s(m) \right )^2 + K b \frac{\a(N)}{N}
\end{align*}
for $K = \sup_N \la \cN_0, m \ra$. Taking into account that $f^b(\cN_t) \geq 0$ and $f^b(\cN_0) \leq \sum_m m \cN_0 \leq K$, we obtain from the martingale \eqref{eq:mart} that
\begin{equation} \label{eq:bigmasses}
\E \left ( \int_0^t \left ( \sum_{m=b}^{\a(N)-1} m \cN_s \right )^2 \ds \right ) \leq 2K \left ( \frac{1}{b} + T \frac{\a(N)}{N} \right ).
\end{equation}
To conclude, take any $f$ with compact support and $C$ be a constant that may change from line to line, but which depends only on $\sup|f|$, $(\cN_0)$ and $T$. Denote
\[
\|g\| = \E \left ( \int_0^T |g(s)| \ds \right ),
\]
the norm on $L^1(\P \otimes \un_{[0,T]} \dt)$. To begin with, \eqref{eq:GN_f} and \eqref{eq:bigmasses} show that
\begin{equation} \label{eq:inegGNf}
\begin{split}
\left \| \GN f(\cN_s) - \frac12 \sum_{m,m'=1}^b \right. & \left.  \vphantom{\sum_{m,m'=1}^b} \left ( f(m+m') - f(m) - f(m') \right ) m m' \cN_s(m) \cN_s(m') \right \| \\
& \leq C \left ( \frac{\a(N)}{N} + \frac{1}{b} \right ).
\end{split}
\end{equation}
Since $\cN \to c$ uniformly on compact sets, then
\[
\frac12 \sum_{m,m'=1}^b \left ( f(m+m') - f(m) - f(m') \right ) m m' \cN(m) \cN(m')
\]
converges to
\[
\frac12 \sum_{m,m'=1}^b \left ( f(m+m') - f(m) - f(m') \right ) m m' c(m) c(m')
\]
for $\|\cdot\|$. On the other hand, the quadratic variation of $\MNf_T$ is
\[
\la \MNf_T \ra = \sum_{t \leq T} \left ( \Dl \MNf_t \right )^2
\]
where $\Dl \MNf_t$ is the jump of $\MNf$ at $t$, which is clearly bounded by $3 \sup_m |f(m)| / N$. The number of jumps on $[0,T]$ is of order $N$ as in the proof of Lemma \ref{lem:tightness}, so that $\E(\la \MNf_T \ra) \to 0$. By Doob's inequality, we thus have
\[
\E \left ( \left ( \sup_{0 \leq t \leq T} \MNf_t \right )^2 \right ) \to 0
\]
so that
\[
\| \MNf \| \to 0.
\]
Passing to the limit for $\| \cdot \|$ in \eqref{eq:mart} and using \eqref{eq:inegGNf} ensure that
\[
\left \| f(c_{\cdot}) - f(c_0) - \int_0^{\cdot} \frac12 \sum_{m,m'=1}^b \left ( f(m+m') - f(m) - f(m') \right ) m m' c_s(m) c_s(m') \ds \right \| \leq C \frac1b.
\]
Having $b \to \infty$ shows that $c$ solves \eqref{eq:smolu2} a.s., for almost every $t \in [0,T]$. By continuity, $c$ solves \eqref{eq:smolu2} a.s. on $[0,T]$, and the proof is complete.

\subsection{About the total concentration of particles}

Let us insist that Theorem~\ref{th:convconc} does not allow to bypass the proof of the second part of Theorem~\ref{th:conv}. Indeed, the concentrations $(c_t(m))$ converge to the solution of Smoluchowski's equation by Theorem \ref{th:convconc}, but Theorem~\ref{th:convconc} only states a convergence in $\ell^1$. Therefore, at this point, the convergence of the total mass $\sum_{m=1}^{\a(N)-1} m \cN_t(m)$ towards $\sum_{m \ge 1} mc_t(m) = n_t$, given in \eqref{eq:nt}, remains unproven. 

Even though we do not have yet a precise estimate on the number of particles in solution past gelation time, we are however at least able to show that on any compact interval, w.h.p this total mass is uniformly bounded below. This will turn out to be useful later on when having to use some asymptotic results. 

\begin{lemma} \label{lem:posconc}
For any $t \geq 0$, there exists $\nu_t > 0$ such that
\[
\P( \nN_t \leq \nu_t) \to 0.
\]
\end{lemma}

\begin{proof}
Fix a time $t \geq 0$, and denote $X^{(N)}_i$ the indicator function of the event that no link with an end in $i$ is activated at time $t$. Since $N-1$ links start from $i$, and each is activated independently with probability $1 - e^{-t/N}$, $X^{(N)}_i$ is 1 with probability $e^{-t (N-1)/N}$ and 0 with probability $1-e^{-t (N-1)/N}$. Moreover, the $X^{(N)}_i$ have a small covariance, since indeed, $X^{(N)}_i$ and $X^{(N)}_j$ are both $0$ when the $2 N - 3$ links with an end in $i$ or $j$ are not activated. These $2 N - 3$ activations are independent, so
\[
\Cov (X^{(N)}_i,X^{(N)}_j) = e^{-t (2N-3)/N} - e^{-2 t (N-1) / N} \sim \frac{t}{N} e^{-2 t}.
\]
It is then a simple application of Chebyshev's inequality to check that, for $\nu_t < e^{-t}$,
\[
\P \left ( \sum_{i = 1}^N X^{(N)}_i < \nu_t N \right ) \to 0.
\]
Hence, w.h.p., more than $\nu_t N$ particles are linked to no other particle, and are thus in solution.
\end{proof}

\section{Erd\H{o}s-R\'enyi random graph} \label{sec:ER}

The configuration of particles in solution and the Erd\H{o}s-R\'enyi random graph with appropriate parameters, conditioned on having no large component, are very closely linked. Until the first gelation event, this is obvious; afterwards it is much less so. However, we will see in Sections~\ref{sec:struct} and \ref{sec:altmodel} that a precise relationship between the two holds, at least under assumption \eqref{eq:alpha}. This is why we will start by recalling and proving properties of Erd\H{o}s-R\'enyi graphs.   

\subsection{Known results, consequences on the first gelation time}
Recall from the introduction that we denote by $ER(N,p_N)$ the Erd\H{o}s-R\'enyi random graph with $N$ vertices and connection probability $p_N$. We define $|C_{\max}|$ to be the largest size of its components. As is famously known, $ER(N,p_N)$ exhibits a phase transition.
\begin{itemize}
\item If $N p_N \to \l < 1$, then w.h.p. $|C_{\max}|= O(\log N)$ (subcritical regime). 
\item If $N p_N \to \l > 1$, then there is w.h.p. a unique largest component of size $\zeta(\l) N + o(N)$, with $\zeta(\l)$ the survival probability of a Galton-Watson process with $\cP(\l)$ offspring distribution; and, w.h.p, all other components have size $O(\log N)$ (supercritical regime).   
\item If $N p_N \to 1$, then $|C_{\max}|$ depends more precisely on the behavior of $N p_N$. Loosely speaking, this is the critical regime. More precisely, the {\em critical window} corresponds to $p_N = 1/N  + t/N^{4/3} + o(1/N^{4/3}), t \in \mathbb{R}$, and for such $p_N$, the largest components are of size of order $N^{2/3}$. On the other hand when $p_N \to 1$ but $|N p_N - 1| \gg N^{-1/3}$, we will rather speak of a near-critical regime. We will also sometimes use the wording {\em slightly subcritical} when $N p_N \to 1$ but $N^{1/3}(N p_N - 1) \to - \infty$ (in such regime w.h.p. $\log N \ll |C_{\max}| \ll N^{2/3}$); and {\em slightly supercritical} when $N p_N \to 1$ but $N^{1/3}(N p_N - 1) \to + \infty$ (in such regime w.h.p. $N^{2/3} \ll |C_{\max}| \ll N$). 
\end{itemize}
The transition phase was first proved in the seminal paper by Erd\H{o}s and R\' enyi \cite{ER2}, while the critical window was exhibited in the beautiful paper \cite{AldousCriticalRG}. More detailed results and modern proofs can be found e.g. in  \cite{BollobasRG,JansonRG,vdH}. 

The transition phase result already provides a very interesting result about our model, namely that, under the assumption $ \log (N) \ll \a(N) \ll N$,  
 \begin{equation} \label{eq:Tgel0}
\tau_1\eN \underset{N \to \infty}{\stackrel{\P}{\longrightarrow}} 1, 
\end{equation}
where we recall that $\tau_1\eN$ is the first gelation time, that is, the time when the largest component reaches a size $\ge \a(N)$. 

To see why this holds, let us compare Smoluchowski's discrete model with Flory's. Recall that in Flory's, the configuration at time $t$ is exactly that of $ER(N,\pN_t)$. Further observe that the two algorithms start behaving differently at the first time when a link is activated between a small and a large component, so that the two models exactly coincide at least up to time $\tau_1\eN$. Because $N \pN_t \to t$, by the aforementioned transition phase results,  at a time $t < 1$, all the components in Flory's discrete model are w.h.p. of size $O(\log N) \ll \a(N)$; conversely at a time $t > 1$, w.h.p. there exists at least a component of size $\zeta(t) N+o(N) \gg \a(N)$. Therefore a component of size $\a(N)$ is created w.h.p. in the interval $(1-\eps,1+\eps)$ in Flory's discrete model, and thus in Smoluchowski's discrete  model as well. This just means \eqref{eq:Tgel0}.

Knowing more precisely when the first gelation occurs, and what happens at this time, depends in fact on our choice of threshold. Whenever $\log(N) \ll \a(N) \ll N^{2/3}$, the first gelation event happens before  the critical window is reached. In this {\em near-subcritical} regime there are many components of size comparable to the largest. Not only does it become challenging to control precisely the size of the falling component, but it also turns out to be harder to deal with the combinatorial structure of the remaining ones.  

When, on the other hand, $N^{2/3} \ll \a(N) \ll N$, the first gelation event occurs after we passed the critical window, when a {\em near-supercritical} regime is reached. This regime is easier and much better understood than the near-subcritical one. Indeed, in the near-supercritical regime, there already is w.h.p. a unique component of maximal size (the emerging giant), whereas other components are much smaller. Also, it is easy to control precisely the size of the emerging giant. In fact we have the following result, first proved in \cite{BollobasERG,LuczakCB} (see also \cite{NachmiasPeres,JansonLuczak}) : if $p_N = (1+ \eps_N)/N$ with $N^{-1/3} \ll \eps_N \ll 1$ (so that we are in the slightly supercritical regime), then $ER(N,p_N)$ has a unique largest component of size $2 N \eps_N + o(N \eps_N)$ w.h.p. 

By the exact same reasoning as for proving \eqref{eq:Tgel0}, the above result tells us not only that 
\begin{equation} \label{eq:Tgel}
\tau\eN_1 = 1 + \frac{\a(N)}{2N} + o \left ( \frac{\a(N)}{N} \right ),
\end{equation}
but also that the first falling component is of size $\a(N)+o(\a(N))$, w.h.p. 

Later on, we will need to carry out this argument over several gelation events. We thus need to control the probabilities involved in a uniform manner, and the price to pay is that we have to replace the assumption on $N^{2/3} \ll \a(N) \ll N$ by \eqref{eq:alpha}.

\subsection{Largest component of a slightly supercritical ER graph: precise bounds} \label{sec:largecomp}

\subsubsection{Result}

Let $\cC_1(n,p)$ and $\cC_2(n,p)$ be the largest and second largest components of a $ER(n,p)$ graph (chosen uniformly at random if there are several choices), and denote $|\cC_1(n,p)| \ge |\cC_2(n,p)|$ their sizes.

For simplicity, in the whole Section \ref{sec:largecomp}, we fix two sequences
\begin{equation} \label{eq:epsn}
n^{-1/3} \ll \eps_n^- \leq \eps_n^+ \ll 1.
\end{equation}
These sequences are just technical artifacts to allow for universal constants. We will then consider $(\eps_n)$, $(p_n)$ and $\g$ with
\begin{equation} \label{eq:pn}
\eps_n^- \leq \eps_n \leq \eps_n^+, \quad \g \in [-1,1], \quad p_n = \frac1n (1 + \g \eps_n).
\end{equation}
The main result of this section is the following.

\begin{theorem} \label{th:sizecomp}
Let $(p_n)$ be as in \eqref{eq:pn}. Then, for any $\dl > 0$, there exists a constant $\k > 0$, depending only on $\dl$ and $(\eps_n^{\pm})$, such that,
with probability greater than
\[
1 - n \exp (- \k n \eps_n^3),
\]
it holds that:
\begin{itemize}
\item if $\g > 0$,
\[
(2 \g - \dl) n \eps_n < |\cC_1(n,p_n)| < (2 \g + \dl) n \eps_n, \quad |\cC_2(n,p_n)| < \dl n \eps_n; 
\]
\item if $\g \leq 0$,
\[
|\cC_1(n,p_n)| < \dl n \eps_n.
\]
\end{itemize}
\end{theorem}

Of course, this result is only useful when $n e^{- \k n \eps_n^3} \to 0$. This is why we shall assume \eqref{eq:alpha}, so that, with $\eps_n = \a(n)/n$, we indeed have $n e^{- \k n \eps_n^3} \to 0$, even faster than any power of $n$. 

The results concerning the size of the second largest component, or the case $\g \le 0$ can certainly be significantly improved, but they suffice for our purposes, as we only need to know whether the size of a component exceeds the threshold.

The proof essentially relies on getting good estimates for the exploration process. This is detailed in the following paragraph.

\subsubsection{Exploration process}

The exploration process $(S_k)$ of the $ER(n,p)$ random graph is defined in \cite{vdH}, Chapter 4. Its distribution is given by
\[
S_0 = 0, \quad S_k = S_{k-1} + X_k - 1, \quad X_k \sim \Bin(n-k-S_{k-1}, p), \quad k \geq 1.
\]
Clearly, it can only go up, stay put, or go down by $-1$. An \emph{excursion} of this process will mean an excursion above the current minimum, i.e. the parts of the process between $0$ and $-1$, $-1$ and $-2$, and so on. The important property of the exploration process is that the size of these excursions is exactly the size of the connected components (and a fortiori, the number of excursions is the number of connected components). We will show the following result. 

\begin{prop} \label{prop:boundS}
Define $(S_k)$ by
\begin{equation} \label{eq:Sk}
S_0 = 0, \quad S_k = S_{k-1} + X_k - 1, \quad X_k \sim \Bin(n-k-S_{k-1}, p_n), \quad k \geq 1,
\end{equation}
with $(p_n)$ as in \eqref{eq:pn}. Then, for any $\eta > 0$, there is a constant $\k > 0$, depending only on $\eta$ and $(\eps_n^{\pm})$, such that
\[
\P \left ( \sup_{0 \leq k \leq 3 n \eps_n} \frac{1}{n \eps_n^2} \left | S_k - k \left ( \g \eps_n - \frac{k}{2 n} \right ) \right | > \eta \right ) \leq \exp (- \k n \eps_n^3).
\]
\end{prop}

If we rescale $S$ by letting $S\en_u = S_{\lfloor u n \eps_n \rfloor}/(n \eps_n^2)$, then this can be reformulated as
\[
\P \left ( \sup_{u \in [0,3]} \left | S\en_u - u \left ( \g - \frac{u}{2} \right ) \right | > \eta \right ) \leq \exp (- \k n \eps_n^3).
\]
In other words, with extremely high probability, $(S\en)$ remains in a tube of small vertical section $2\eta$ around the parabola $u(\g - u/2)$. This will be the main ingredient of the proof of Theorem \ref{th:sizecomp}. The force of this result is that it allows to control the exploration process uniformly. In the notation of the proof, one could also decide to optimize the size of the tube $(T(n))$, so that it tends to 0. Keeping uniform probabilities would then require stronger assumptions on $(\eps^{\pm})$. The above result will prove the most convenient to us.

\begin{proof}

\begin{enumerate}[fullwidth]

\item To begin with, replace $(S_k)$, defined in \eqref{eq:Sk}, by an actual random walk: fix a sequence $(\b(n))$ and define
\begin{equation} \label{eq:Sk0}
S_0 = 0, \quad S_k = S_{k-1} + X_k - 1, \quad X_k \sim \Bin(n - k - \b(n), p_n), \quad k \geq 1.
\end{equation}
If $\cF_k$ is the $\s$-algebra generated by $X_1, \dots, X_k$, it is easy to compute that, for $\l \in \R$,
\[
\E \left ( e^{\l S_k} \middle | \cF_{k-1} \right ) = e^{\l S_{k-1}} e^{- \l} \left ( 1 - p_n + p_n e^{\l} \right )^{n - \b(n) - k}.
\]
One then readily checks that $(M_k(\l))$ is a martingale, where
\[
M_k(\l) := e^{\l S_k} / m_k(\l), \quad m_k(\l) := e^{- \l k} \left ( 1 - p_n + p_n e^{\l} \right )^{(n - \b(n))k - k(k+1)/2}.
\]
Recall that $(p_n)$ is as in \eqref{eq:pn}, which is far too strong for most computations, but weaker assumptions would not be much more useful. We also assume that $\l := \l_n$ changes with $n$, with $|\l_n| \leq 1$ and $\l_n \to 0$. These assumptions ensure that, in the following computations, the constants hidden into the $O(\cdot)$ only come from the Taylor expansion of usual functions, or the relationship between the different sequences, and thus depend only on $(\eps_n^{\pm})$. It is first easy to check that
\[
\log \left ( 1 - p_n + p_n e^{\l_n} \right ) = p_n \l_n \left ( 1 + O (\l_n) \right )
\]
and that, for $k \leq 3 n \eps_n$, say, we have 
\begin{equation} \label{eq:logmk}
\frac{1}{\l_n} \log m_k(\l_n) = k  \left (\g \eps_n - \frac{k}{2 n} \right ) + O (\b(n) \eps_n) + O (n \eps_n^3) + O (n \eps_n \l_n).
\end{equation}
Assume now that $\l_n \geq 0$. For any nonnegative sequence $T(n)$, Doob's inequality implies that
\[
\P \left ( \sup_{k = 0, \dots, 3 n \eps_n} M_k(\pm \l_n) > e^{\l_n T(n)} \right ) \leq \E \left ( M_{3 n \eps_n} (\pm \l_n) \right ) e^{- \l_n T(n)} = e^{- \l_n T(n)}.
\]
Hence, with probability at least $1 - e^{- \l_n T(n)}$, we have
\[
\frac{1}{\l_n} \log M_k(\pm \l_n) \leq T(n), \quad k = 0, \dots, 3n \eps_n.
\]
In addition,
\begin{align*}
\frac{1}{\l_n} \log M_k(\pm \l_n) & = \pm S_k - \frac{1}{\l_n} \log(m_k(\pm \l_n)) \\
& = \pm \left ( S_k - k  \left (\g \eps_n - \frac{k}{2 n} \right ) \right ) + O (\b(n) \eps_n) + O (n \eps_n^3) + O (n \eps_n \l_n),
\end{align*}
where the second equality comes from \eqref{eq:logmk}. Hence for some constant $C$ depending only on $(\eps_n^{\pm})$, with probability greater than $1 - 2 e^{- \l_n T(n)}$, for $k \leq 3 n \eps_n$,
\begin{equation} \label{eq:bound1}
\left | S_k  - k \left ( \g \eps_n - \frac{k}{2n} \right ) \right | \leq C \left ( \b(n) \eps_n + n \eps_n^3 + n \eps_n \l_n \right ) + T(n).
\end{equation}
Let us now fix $\eta \in (0,1/2)$ and take
\[
 T(n) = \frac{\eta}{4} n \eps_n^2, \quad \l_n  = \frac{\eta}{4 C} \eps_n \wedge 1.
\]
Then \eqref{eq:bound1} reads: there is a constant $C$, depending only $(\eps_n^{\pm})$ such that, for $k \leq 3 n \eps_n$,
\begin{equation} \label{eq:bound2}
\frac{1}{n \eps_n^2} \left | S_k  - k \left ( \g \eps_n - \frac{k}{2n} \right ) \right | \leq C  \left ( \frac{\b(n)}{n \eps_n} + \eps_n  \right ) + \eta/2,
\end{equation}
with probability greater than $1 - 2 e^{- \k n \eps_n^3}$, where $\k = \eta^2 / (16 C) > 0$. At least, this holds for $n$ large enough, uniformly in the parameters, and we may take a smaller $\k$ to make the result true for all $n$.

\item Define $(\Sp_k)$ as in \eqref{eq:Sk0} with $\b(n) = 0$. From \eqref{eq:bound2}, with probability greater than $1 - 2 e^{- \k n \eps_n^3}$,
\begin{align*}
\frac{1}{n \eps_n^2} \sup_{k = 0, \dots, 3 n \eps_n} \Sp_k & \leq \frac{1}{n \eps_n^2} \sup_{k = 0, \dots, 3 n \eps_n} k \left ( \g \eps_n - \frac{k}{2n} \right ) + C \eps_n + \eta/2 \\
& \leq \frac{\g^2}{2} + C + \eta/2.
\end{align*}
Hence, with probability greater than $1 - 2 e^{- \k n \eps_n^3}$, 
\begin{equation} \label{eq:sup}
\sup_{k = 0, \dots, 3 n \eps_n} \Sp_k \leq (C + 1) n \eps_n^2.
\end{equation}

\item Consider finally the real exploration process $(S_k)$ defined by \eqref{eq:Sk}. Define $(\Sp_k)$ as in \eqref{eq:Sk0} with $\b(n) = 0$, and $(\Sm_k)$ similarly with
\[
\b(n) = (C + 1) n \eps_n^2.
\]
It is clear that we can couple $S$, $\Sp$, $\Sm$, using for instance Bernoulli variables, in a way that
\[
\Sm_k \leq S_k \leq \Sp_k, \quad k = 0, \dots, 3 n \eps_n,
\]
as long as \eqref{eq:sup} holds. But \eqref{eq:sup} holds with probability greater than $1 - 2 e^{- \k n \eps_n^3}$, so \eqref{eq:bound2} shows that, with probability at least $1 - 4 e^{- \k n \eps_n^3}$, we have
\[
\frac{1}{n \eps_n^2} \left | S_k  - k \left ( \g \eps_n - \frac{k}{2n} \right ) \right | \leq C  \left ( (C + 1) \eps_n + \eps_n \right ) + \eta/2 \leq C (C + 2) \eps_n^+ + \eta/2,
\]
for $k \leq 3 n \eps_n$. For $n$ large enough, depending only on $(\eps_n^+)$, the right hand side above becomes smaller than $\eta$. Choosing a smaller $\k$ allows to take care of smaller values of $n$ and get rid of the constant $4$, which yields Proposition \ref{prop:boundS}. 
\end{enumerate}

\end{proof}

\subsubsection{Proof of Theorem \ref{th:sizecomp}}

Assume $\g \geq 0$, fix $\dl > 0$ and $\eta > 0$, and let $P^{\pm}(u) = u(\g - u / 2) {\pm} \eta$. Define $u^- \leq U^-$ the zeros of $P^-$, with the convention that $u^- = U^- = 0$ whenever $P^-$ has no roots. Take also $U^+$ the largest solution to $P^+(u) = - \eta$, see Figure \ref{fig:parabola}. Note that, always, $0 \leq u^- \leq U^- < U^+$.
  
\begin{figure}[htb]
\centering
\includegraphics[width=0.5 \columnwidth]{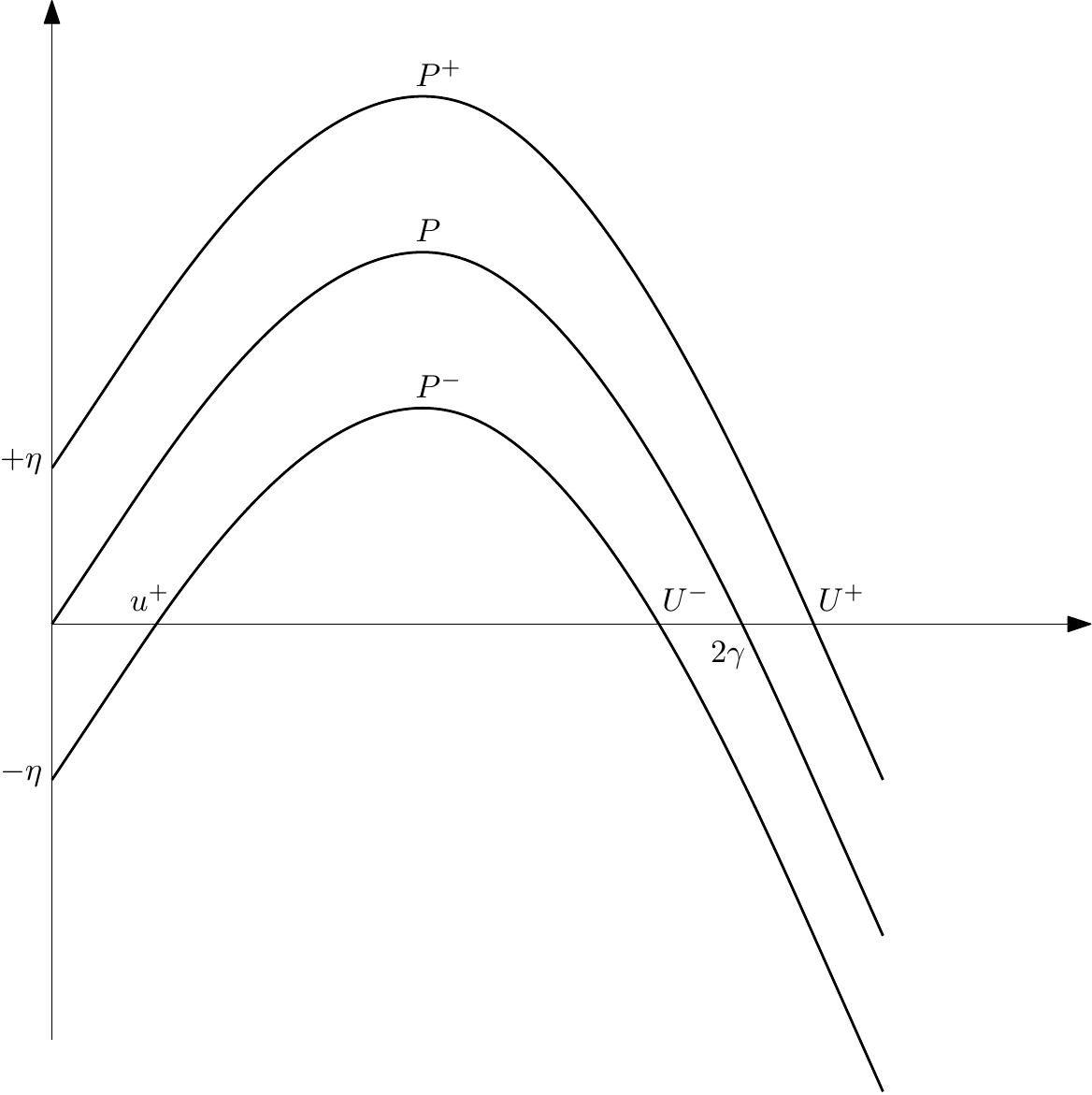}
\caption{The exploration process is wedged between $P^-$ and $P^+$.}
\label{fig:parabola}
\end{figure}

Clearly, we may choose $\eta$ small enough, independently of $\g \in [0,1]$, such that $2 \g - \dl / 2 < U^- \leq 2 \g < U^+ < 2 \g + \dl / 2$ and $u^- < \dl / 2$. Let $(S_k)$ be the exploration process of $ER(n,p_n)$. According to Proposition \ref{prop:boundS}, there is a $\k > 0$, depending only on $\eta$ and $(\eps_n^{\pm})$ such that, with probability greater than $1 - 2 e^{- \k n \eps_n^3}$, we have
\[
\sup_{0 \leq k \leq 3 n \eps_n} \frac{1}{n \eps_n^2} \left | S_k - k \left ( \g \eps_n - \frac{k}{2 n} \right ) \right | < \eta.
\]
Extending the trajectory of $S$ to $\R^+$ by linear interpolation and defining
\[
S^{(n)}_u = \frac{1}{n \eps_n^2} S_{u n \eps_n},
\]
this can be rewritten as
\begin{equation} \label{eq:boundSn}
P^-_u < S^{(n)}_u < P^+_u, \quad u \in [0,3],
\end{equation}
with probability greater than $1 - e^{- \k n \eps_n^3}$. Clearly, this implies that there is an excursion (above the current minimum) of $S^{(n)}$, starting in $[0,u^-]$, and ending in $[U^-,U^+]$, which has thus a size between $U^- - u^- > 2 \g - \dl$ and $U^+ < 2 \g + \dl$. Hence, $S$ itself has an excursion of size between $(2 \g - \dl) n \eps_n$ and $(2 \g + \dl) n \eps_n$, and thus the random graph has a component $\cC_0$ of this size.

To bound the size of the other components, note first that any component explored before $\cC_0$ has size less than $u^- n \eps_n \leq \dl n \eps_n/2$. Let $V$ be the (random) set of vertices explored after $\cC_0$, let $\cC(v)$ be the component containing a vertex $v$, and $|\cC(v)|$ its size. Consider
\[
M = \sum_{v \in V} \unn{|\cC(v)| > \dl n \eps_n}.
\]
By exchangeability, all $|\cC(v)|, v \in V$ have the same distribution, so that $\E(M) \leq n \P(|\cC(v_0)| > \dl n \eps_n)$, where $v_0$ is the first vertex explored after $\cC_0$. The corresponding excursion starts somewhere in $[U^-,U^+]$. It is clear that the largest excursion that can fit between $P^-$ and $P^+$ has length less than $U^+ - U^- < \dl$. Hence,
\[
\P(|\cC(v_0)| > \dl n \eps_n) \leq 2 \exp (- \k n \eps_n^3).
\]
The probability that there is an excursion of size greater than $\dl n \eps_n$ explored after the large component is thus
\[
\P(M \geq 1) \leq \E(M) \leq 2 n \exp (- \k n \eps_n^3),
\]
and the result follows, taking again a smaller $\k$ if necessary to absorb the constants.

Finally, the result for $\g \leq 0$ is a direct consequence of the case $\g = 0$. Indeed the size of the largest component is stochastically increasing in $\g$, so that, for $\g \leq 0$,
\[
\P \left ( |\cC_1(n, (1 + \g \eps_n)/n)| \geq \dl n \eps_n \right ) \leq \P \left ( |\cC_1(n, 1/n)| \geq \dl n \eps_n \right ) \leq 1 - n \exp (- \k n \eps_n^3)
\]
and this completes the proof.

\subsection{First gelation event in a dynamic graph: precise bounds} \label{sec:cGN}

Observe that in Theorem \ref{th:sizecomp} the probability of presence of an edge is fixed, and we control precisely the size of the largest (and sometimes second-largest) component. 

Until the first gelation event in our model, the probability of presence of an edge increases exactly like $\pN_t$. Hence, until gelation, we are exactly dealing with the usual dynamic version of an ER graph $(ER(N,\pN_t))_{t \ge 0}$, and can deduce approximately at which time a large component is formed. This is exactly how we argued to obtain \eqref{eq:Tgel}. 

But as we already mentioned, ER graphs and our model will turn out to remain closely related even at times past gelation. After gelation, edges are still created at rate $1/N$, the threshold remains $\a(N)$, but we should take into account that one or several components have fallen so the total number of vertices is, say, $n$, less than $N$.

It therefore makes sense to consider the c\`adl\`ag random graph process $(\cGN_t(n))_{t \geq 0}$ on $[n]$ obtained by creating each link independently at rate $1/N$. We define
\begin{itemize}
\item $\sN(n)$ as the time when a large component i.e. a component of size $\geq \a(N)$, appears;
\item $\gN(n)$ as the size of that large component;
\item $\esN(n) = n - \gN(n)$ as the number of particles remaining in solution at $\sN(n)$.
\end{itemize}
We claim the following. Beware in the formula of the difference between $\sN_{\pm}$ and $\vsN_{\pm}$, so that $\sN_- \leq \sN_+$, and $\vsN_- \leq \vsN_+$.

\begin{prop} \label{prop:atgel}
Assume that $(\a(N))$ verifies \eqref{eq:alpha} holds. Let $\nu \in (0,1]$, $\dl > 0$,
\[
\sN_{\pm}(n) = - N \log \left( 1 - \frac1n \left ( 1 + \frac12 \left ( 1 \pm \dl \right ) \frac{\a(N)}{n} \right )\right ),
\]
and
\[
\vsN_{\pm}(n) = - N \log \left( 1 - \frac1n \left ( 1 - \frac12 \left ( 1 \mp \dl \right ) \frac{\a(N)}{n} \right )\right ).
\]
Then there is a constant $\k > 0$, depending only on $\dl$, $\nu$ and $(\a(N))$, such that, for all $N \geq 2$ and $\nu N \leq n \leq N$, the following statements hold with probability greater than
\[
1 - N \exp  \left ( - \k \frac{\a(N)^3}{N^2} \right ).
\]
\begin{enumerate}
\item The gelation time verifies
\[
\sN_-(n) \leq \sN(n) \leq \sN_+(n).
\]
\item The size of the large component enjoys
\[
\a(N) \leq \gN(n) \leq (1+\dl) \a(N).
\]
\item The gelation time also verifies
\[
\vsN_- \left ( \esN(n) \right ) \leq \sN(n) \leq \vsN_+ \left ( \esN(n) \right ).
\]
\item Conditionally on $\esN(n)$ and $\sN(n)$, the graph $ER \left ( \esN(n), \pN_{\sN(n)} \right)$ has no large component.
\end{enumerate}
\end{prop}

This result gives precise bounds on the gelation time for $\cGN(n)$. Point 1 gives a bound in terms of $n$, which is natural, whereas the bound in Point 3 is in terms of $s\eN(n)$, the number of particles in solution right after the creation of the large component. Both results will turn out to be useful for us. Point 2 allows to control the size of the falling components, whereas Point 4 will allow us to compare our original model to an alternative model constructed from the graphs $\cGN(n)$, see Section \ref{sec:altmodel}.

\begin{proof}
We may obviously assume $\dl < 1/2$. Let
\[
\eps_n^- = \frac{1}{n} \inf_{n \leq k \leq n / \nu} \a(k), \quad \eps_n^+ = \frac{1}{n} \sup_{n \leq k \leq n / \nu} \a(k), \quad \eps_n = \frac{\a(N)}{n}.
\]
Since $(\a(N))$ verifies \eqref{eq:alpha0}, then \eqref{eq:epsn} holds, and we may take $\k$ as in Theorem \ref{th:sizecomp}, which depends only on $\dl$ and $(\eps_n^{\pm})$, and thus on $\dl$, $\nu$ and $(\a(N))$. We shall write w.h.p. below to mean with probability greater than
\[
1 - C n \exp (- \k n \eps_n^3)
\]
where $C$ is just some large enough constant to take into account the union bounds below.

At any time $t$, the distribution of $\cGN_t(n)$ is that of a $ER(n,\pN_t)$. Hence $\sN(n) < \sN_-(n)$ means that $ER(n,\pN_{\sN_-(n)})$ has a large component, with
\[
\pN_{\sN_-(n)} = \frac1n \left ( 1 + \frac12 \left( 1 - \dl \right ) \frac{\a(N)}{n} \right ) = \frac1n \left ( 1 + \g \eps_n \right )
\]
for $\g = 1/2 (1 - \dl)$. But Theorem \ref{th:sizecomp} implies that
\[
\left | \cC_1 \left ( n, \pN_{\sN_-(n)} \right ) \right | < (2 \g + \dl) n \eps_n = \a(N)
\]
w.h.p. Hence $\sN(n) > \sN_-(n)$ w.h.p. Conversely, $\sN(n) > \sN_+(n)$ means that the graph $ER(n,\pN_{\sN_+(n)})$ has no large component, what, for the same reason, does not happen w.h.p. This proves Point 1.

For Point 2, notice that $\cC_1(n,p)$ is stochastically increasing with $p$. Hence, on the event of high probability $\{ \sN(n) < \sN_+(n) \}$, the size of the large component at $\sN(n)$ is bounded by $|\cC_1(n,\pN_{\sN_+(n)})|$, which, by Theorem \ref{th:sizecomp}, is bounded by $(1 + 2 \dl) n \eps_n$ w.h.p., and Point 2 follows.

Now, using the two results just proved, let us write that, still w.h.p.,
\begin{align*}
\pN_{\sN(n)} & \leq \pN_{\sN_+(n)} \\
& = \frac1n \left ( 1 + \frac12 \left( 1 + \dl \right ) \frac{\a(N)}{n} \right ) \\
& = \frac{1}{\esN(n)} \frac{\esN(n)}{n} \left ( 1 + \frac12 \left( 1 + \dl \right ) \frac{\a(N)}{n} \right ) \\
& \leq \frac{1}{\esN(n)} \left ( 1 - \frac{\a(N)}{n} \right ) \left ( 1 + \frac12 \left( 1 + \dl \right ) \frac{\a(N)}{n} \right ) \\
& = \frac{1}{\esN(n)} \left ( 1 - \frac12 \left( 1 - \dl \right  ) \frac{\a(N)}{n} - \frac12 \left( 1 + \dl \right ) \frac{\a(N)^2}{n^2} \right ) \\
& \leq \frac{1}{\esN(n)} \left ( 1 - \frac12 \left( 1 - \dl \right  ) \frac{\a(N)}{n} \right ).
\end{align*}
But obviously, $\esN(n) \geq n - 2 \a(N)$, so that
\[
\frac{\a(N)}{n} = \frac{\a(N)}{\esN(n)} \frac{\esN(n)}{n} \geq  \frac{\a(N)}{\esN(n)} \left ( 1 - 2 \frac{\a(N)}{n} \right ) \geq \frac{\a(N)}{\esN(n)} \left (1 - 2 \frac{\a(N)}{\nu N} \right ) \geq \frac{\a(N)}{\esN(n)} (1 - \dl),
\]
for $N$ large enough, since $\a(N) / N \to 0$. Finally
\begin{align*}
\pN_{\sN(n)} & \leq \frac{1}{\esN(n)} \left ( 1 - \frac12 \left( 1 - \dl \right ) (1 - \dl) \frac{\a(N)}{\esN(n)} \right ) \\
& \leq \frac{1}{\esN(n)} \left ( 1 - \frac12 \left( 1 - 2 \dl \right ) \frac{\a(N)}{\esN(n)} \right ),
\end{align*}
and the result follows after reordering. The other direction is similar.

For the last point, let us condition on $\esN(n)$ and $\sN(n)$. On the event of high probability $\{\sN(n) \leq \vsN_+(\esN(n))\}$, we have
\[
\pN_{\sN(n)} \leq \pN_{\vsN_+(\esN(n))} = \frac{1}{\esN(n)} \left ( 1 - \frac12 \left ( 1 - \dl \right ) \frac{\a(N)}{\esN(n)} \right ) \leq \frac{1}{\esN(n)}.
\]
By the second point of Theorem \ref{th:sizecomp},
\[
\left | \cC_1 \left ( \esN(n),\frac{1}{\esN(n)} \right ) \right | \leq \dl n \eps_n < \a(N)
\]
with probability at least
\[
1 - \esN(n) \exp (- \k \esN(n) \eps_n^3). 
\]
Hence, by stochastic domination again,
\begin{align*}
\P \left ( \left | \cC_1 \left ( \esN(n), \pN_{\sN(n)} \right | \right ) \geq \a(N) \right ) & \leq \esN(n) \exp (- \k \esN(n) \eps_n^3) + \P \left ( \sN(n) > \vsN_+(\esN(n)) \right ) \\
& \leq \esN(n) \exp (- \k \esN(n) \eps_n^3) + C n \exp (- \k n \eps_n^3).
\end{align*}

To conclude, notice $N \geq sN(n) \geq n - 2 \a(N) \geq \nu N /2$ for $N$ large enough, and that for $\nu N /2 \leq r \leq N$,
\[
r \exp (- \k r \eps_n^3) = r \exp \left ( - \k r \frac{\a(N)^3}{n^3} \right ) \leq N \exp \left ( - \frac{\k \nu}{2} \frac{\a(N)^3}{N^2} \right ).
\]
As usual, it suffices to take $\k$ smaller to get rid of the ``N large enough'' and the constants.
\end{proof}

\section{Combinatorial structure} \label{sec:struct}

The two results below explain the combinatorial structure of particles in solution in Smoluchowski's discrete model, at a fixed time (Lemma \ref{lem:struct1}), or at some particular stopping times (Lemma \ref{lem:struct2}).   

Let us start with an important preliminary remark. The natural filtration $(\cF_t)$ of Smoluchowski's (resp. Flory's) discrete model is the one generated by the clocks, i.e. for any $t \ge 0$, 
\[
\cF_t : = \sigma \left ( e_{ij} \unn{e_{ij} \le t}, \; i,j \in [N] \right ),
\]
where $e_{ij}$ is the clock on the link between $i$ and $j$.

In Flory's discrete model the presence of the link between any two particles $i$ and $j$ at time $t$ is exactly equivalent to the fact that the corresponding clock has rung, and in particular it is independent of the presence of any other link. Hence, the configuration in Flory's discrete model at time $t$ is exactly that of $ER(N,\pN_t)$. Similarly, for any $S \subset [N]$, the subgraph consisting particles in $S$ and the links between these particles activated before time $t$ is exactly $ER(|S|,\pN_t)$. In other words, this subgraph is  $\cF_t^S := \sigma \{e_{ij} \unn{e_{ij} \le t}, i,j \in S \}$-measurable.  

In Smoluchowski's discrete model on the other hand, for a subset $S \subset [N]$, the knowledge of all the clocks $S \lra S$, or even the knowledge of all clocks $S \lra [N]$ is not enough to decide which links $S \lra S$ are created, or not. In fact we need the information on all clocks up to time $t$. For instance, in Figure \ref{fig:example}, knowing all clocks but the one on link 3 does not allow to tell whether 5 is created, or not, when it rings. In some sense, the model is ``non-local''. In other words the configuration of particles in $S$ at time $t$ in Smoluchowski's model is no longer $\cF_t^S$-measurable.

Fortunately, there is a nice consistency property. For $S \subset [N]$, let us define a graph $G_S(t)$ by applying Smoluchowski's algorithm up to time $t$, but only to the subset $S$ (i.e. perform the algorithm when only activating the links $S \lra S$). As we just mentioned, $G_S(t)$ has a priori nothing to do with the configuration on $S$ at $t$. It is however easy to check that, conditionally given that no link in $S \lra \bar{S}$ exists at time $t$, the configuration on $S$ is indeed given by $G_S(t)$.

This observation is the main tool in obtaining the following result. Recall that $S(t)$ is the set of particles in solution at time $t$, and let $ER'(n,p)$ be the $ER(n,p)$ random graph conditioned on having no large cluster, i.e. conditioned on having no cluster of size $\geq \a(N)$. 

\begin{lemma} \label{lem:struct1}
For any $t \geq 0$, conditionally given $S(t)$, the configuration on $S(t)$ is that of a
\[
ER' \left ( |S(t)|, \pN_t \right )
\]
random graph, i.e. a $ER \left ( |S(t)|, \pN_t \right )$ graph conditioned on having no large component.
\end{lemma}

\begin{proof}
Recall that for $S \subset [N]$, and $t\ge 0$, $\cF_t^S = \sigma \left ( \mathbf{e}_{ij}\unn{e_{ij} \le t}, i \in S \mbox{ and } j \in S \right )$, and introduce as well 
\[
\mathfrak{F}_t^S := \sigma \left ( \mathbf{e}_{ij}\unn{e_{ij} \le t}, i \in S \mbox{ or } j \in S \right )
\]
Observe in particular that for any $S,t$, the $\sigma$-fields $\cF_t^S$ and $\mathfrak{F}_t^{\bar{S}}$ are independent, since they are generated by disjoint sets of clocks.  

Furthermore, it is easily seen that $\{S(t)=S\} = E_1^S(t) \cap E_2^S(t)$, where 
\begin{itemize}
\item[$E_1^S(t)$:] $G_S(t)$ has no large component;
\item[$E_2^S(t)$:] $G_{\bar{S}}(t)$ consists only of large components;
and any link $i \lra j$, with $i \in S$, $j \in \bar{S}$, is activated \emph{after} $j$ has become part of a large cluster in $G_{\bar{S}}(t)$.
\end{itemize} 
Obviously $E_1^S(t) \in \cF_t^S$, while $E_2^S(t) \in \mathfrak{F}_t^{\bar{S}}$, so these two events are independent. 

Now, take $G$ a graph on $S$ with no large component. Recall from the introduction to the section that
\[
\P \left ( C(t) = G \middle | S(t) = S \right ) = \P \left ( G_S(t) = G \middle | S(t) = S \right ).
\]
Clearly $\{ G_S(t) = G \} \in \cF_t^S$ so it is also independent from $E_2^S(t)$, and it follows that 
\[
\P \left ( G_S(t) = G \middle | S(t) = S \right ) = \P \left ( G_S(t) = G \middle | E_1(t), \; E_2(t) \right ) = \P \left ( G_S(t) = G \middle | E_1(t) \right ).
\]
This is the result, since conditionally on $E_1(t)$, $G_S(t)$ is just given by creating all activated links, as long as no large component is formed, i.e. it is a $ER'(S,\pN_t)$ graph.
\end{proof}

A simple but useful extension of this result can be obtained.   

\begin{defn}
We say that $\tau$ is a {\em gelation stopping time} if
\begin{itemize} 
\item $\tau$ is a $(\cF_t)$-stopping time,
\item for any $S \subset [N]$ and $t \geq 0$, conditionally given $\{S(t)=S\}$, $\tau \unn{\tau \leq t}$ is independent of $\cF_t^S$.
\end{itemize} 
\end{defn} 

The most important and only gelation stopping times that we will consider are the gelation times $(\tau_k)$. To check that they are indeed gelation stopping times, note that, conditionally given $\{S(t) = S\}$, $\tau_k \unn{\tau_k \leq t}$ is simply determined by the $k$-th gelation time in $(G_{\bar{S}}(s))_{s \leq t}$, and is therefore independent of $\cF_t^S$. On the other hand, for instance, the first time after $\tau_k$ that a cluster of mass $\a(N)/2$ appears in solution is not a gelation stopping time.
 
\begin{lemma} \label{lem:struct2}
For any gelation stopping time $\tau$, conditionally on $\tau$ and $S(\tau)$, the configuration on $S(\tau)$ is that of a
\[
ER' \left ( |S(\tau)|, \pN_{\tau} \right )
\]
random graph.
\end{lemma}

\begin{proof}
Informally, we essentially have the same proof as for a deterministic time:
\begin{align*}
\P \left ( C(\tau) = G \middle | \tau = t, S(\tau) = S \right ) & = \P \left ( G_S(t) = G \middle | \tau = t, S(t) = S \right ) \\
& = \frac{\P \left ( G_S(t) = G, \tau = t \middle | S(t) = S \right )}{\P \left ( \tau = t \middle | S(t) = S \right )} \\
& = \frac{\P \left ( G_S(t) = G \middle | S(t) = S \right ) \P \left ( \tau = t \middle | S(t) = S \right )}{\P \left ( \tau = t \middle | S(t) = S \right )}  \\ 
& = \P \left ( G_S(t) = G \middle | S(t) = S \right ), 
\end{align*}
where, for the third equality, we use that both $G_S(t)$ and $\{ \tau = t\}$ are independent conditionally on $\{S(t) = S\}$. Then the conclusion is as in the proof of the first result.
  
More rigorously, one needs to justify the degenerate conditioning with respect to the event $\{\tau=t\}$, and we thus need to study
\[
\lim_{\eps \to 0} \P \left ( C(\tau) = G \middle | S(\tau) = S, t - \eps \leq \tau \leq t \right ) = \lim_{\eps \to 0} \P \left ( G_S(\tau) = G \middle | S(\tau) = S, t - \eps \leq \tau \leq t \right ).
\]
But
\[
\P \left ( G_S(\tau) = G \middle | S(\tau) = S, t - \eps \leq \tau \leq t \right ) = \frac{\P \left ( G_S(\tau) = G,  t - \eps \leq \tau \leq t \middle | S(\tau) = S \right )}{\P \left (  t - \eps \leq \tau \leq t \middle | S(\tau) = S \right )}.
\]
Recall that we work with $N$ fixed in this section, so that the probability that no clock rings between $\tau$ and $\tau +\eps$ goes to one as $\eps \to 0$. By using the Markov property at $\tau$, it is easily seen that, as $\eps \to 0$,  
\[
\P \left ( G_S(\tau) = G,  t - \eps \leq \tau \leq t \middle | S(\tau) = S \right ) \sim \P \left ( G_S(t) = G,  t - \eps \leq \tau \leq t \middle | S(t) = S \right )
\]
and
\[
\P \left (  t - \eps \leq \tau \leq t \middle | S(\tau) = S \right ) \sim \P \left (  t - \eps \leq \tau \leq t \middle | S(t) = S \right ).
\]
It follows that as $\eps \to 0$,
\[
\P \left ( C(\tau) = G \middle | S(\tau) = S, t - \eps \leq \tau \leq t \right ) \sim \frac{\P \left ( G_S(t) = G,  t - \eps \leq \tau \leq t \middle | S(t) = S \right )}{\P \left (  t - \eps \leq \tau \leq t \middle | S(t) = S \right )}
\]
and we conclude as in the first informal computation, using that $\{ t - \eps \leq \tau \leq t \}$ is independent of $G_S(t) = G$ conditionally on $\{S(t) = S\}$, since $\tau$ is a gelation stopping time.
\end{proof} 

\section{Alternative model} \label{sec:altmodel}

The goal of this section is to present an alternative model which is easier to study than Smoluchowski's since it has a much nicer combinatorial structure. We will prove however that w.h.p. the two models have the same distribution on compact time intervals. 

A direct consequence of  Lemma \ref{lem:struct2} is that, conditionally on $\tau_i$ and $n = |S(\tau_i)|$, the confi\-guration in solution is that of conditioned ER graph $ER'(n,p_{\tau_i})$. Informally speaking, our goal is simply to somehow get rid of the ``no large component'' conditioning part. 

We introduce below a slightly more complicated process, which is however merely built from dynamic ER graphs. Encoded in that process are two processes: our original one (at least, its state in solution $(C(t))$), and a process $(D(t))$ which is much easier to study. In other words, we have a coupling between our process and the more simple $(D(t))$: this is the purely combinatorial Lemma \ref{lem:equivproc}. The fundamental fact is that the processes are actually equal w.h.p., as proved in Lemma \ref{lem:equivasymp}. This second statement makes heavy use of what was proven in Proposition \ref{prop:atgel}. To explain our strategy, we will first start with a simple example.

\subsection{A warming-up example} 

Consider two variables $X$ and $Y$ on $[N]$, constructed as follows:
\begin{itemize}
\item $X$ has a uniform distribution on $[N]$;
\item conditionally on $X$, $Y$ has a uniform distribution on $\{X,\dots,N\}$.
\end{itemize}
Note that the distribution of $(X,Y)$ is not the distribution of an independent couple of variables $(U,V)$ conditioned on $V \geq U$. For instance $\P(X = Y = N) \neq \P(U = V = N)$. More generally, one cannot obtain the distribution of $(X,Y)$ from a mere conditioning of $(U,V)$. To obtain this distribution from independent variables, it is common to use rejection sampling: consider $U$ uniform and an i.i.d. family of uniform variables $(V_k)_{k \geq 1}$, independent from $U$. Denote
\[
K = \inf \{ k \geq 1, V_k \geq U \}.
\]
Then, conditionally on $U$, the distribution of $V_K$ is that of a uniform variable on $\{U, \dots, N\}$. In other words, $(U,V_K)$ has the distribution of $(X,Y)$. Note also that, clearly, the uniform distribution here is not specific.

The whole point of this construction is that it provides a coupling between $(X,Y)$ and $(U,V)$, which allows us to compare them easily. For instance, if the distributions we consider are such that $K = 1$ with high probability, then for any measurable set $E$,
\[
\P((X,Y) \in E) = \P((U,V_K) \in E) \approx \P((U,V) \in E)
\]
where $\approx$ means ``up to $\P(K \neq 1)$'', which is small. Studying $(X,Y)$ then essentially amounts to studying the much more tractable $(U,V)$. We shall carry out this construction in the following section, but in a slightly more elaborate setting.

\subsection{Definition}

Consider a family $(\cGN(n,k))$ indexed by $[N] \times \N$ of independent dynamic random graph processes, with $n \in [N]$ and $k \geq 1$, such that, for each $n$ and $k$, $\cGN(n,k)$ is a random graph process on $[n]$ whose edges appear at rate $1/N$. In other words, $\cGN(n,k)$ has the distribution of $\cGN(n)$ introduced in Section \ref{sec:cGN}. We shall use the same notation $\s$, $g$ and $s$, but writing $\sN(n,k)$, $\gN(n,k)$ and $\esN(n,k)$ to insist that we refer to the process $\cGN(n,k)$. We construct the process $(B(t))$ as follows. 
\begin{description}
\item[\textbf{Step 0}] Let $N_0 = N$, $\s_0 = 0$ and $K(0) = 1$. Consider $\cGN(N_0,K(0))$, up to the time $t_1 = \s(N_0,K(0))$ when a large component, of exact size $g_1 = g(N_0,K(0))$, appears. Define $B(t) = \cGN_t(N_0,K(0))$ for $t \in [0,t_1)$, let $N_1 = N_0 - g_1 = s(N_0,K(0))$ and go to step 1.
\item[\textbf{Step i}] Consider here the graph processes $\cGN(N_i,k)$ for $k \geq 1$. Define $K(i)$ to be the first $k$ such that $\cGN_{t_i}(N_i,k)$ has no large component. Then take $t_{i+1} = \s(N_i,K(i))$ and $g_{i+1} = g(N_i,K(i))$. Finally, define $B(t) = \cGN_t(N_i,K(i))$ for $t \in [t_i,t_{i+1})$, let $N_{i+1} = N_i - g_{i+1} = s(N_i,K(i))$ and go to step $i+1$.
\end{description}
Obviously, we stop and let $t_{i+1} = \pinf$ when no more large component can be created.

We will prove the following result, and, in doing so, explain this construction in more details. Recall that $(C(t))$ is the configuration in solution of our process.

\begin{lemma} \label{lem:equivproc}
The processes $(C(t), t \geq 0)$ and $(B(t), t \geq 0)$ have the same distribution.
\end{lemma}

\begin{proof}
Let $M_0 = N$ and $M_i = c_{\tau_i} N$, $i \geq 1$ to be the mass in solution right after the gelation times. Our process is, before gelation, just a random graph process on $[N]$, so has the distribution of $\cGN(N,1) = \cGN(N_0,K(0))$. In particular, the time of the appearance of a large component and its size are the same for both processes, i.e.
\[
(\tau_1, M_1, (C(t), t \in [0,\tau_1)) \eqlaw (t_1, N_1, (B(t), t \in [0,t_1)).
\]
Let us then prove by induction that for every $i \geq 1$, $\cH_i$ holds, where $\cH_i$ is the assumption
\[
(\tau_i, M_i, (C(t), t \in [0,\tau_i)) \eqlaw (t_i, N_i, (B(t), t \in [0,t_i)).
\]
We just checked $\cH_1$, so assume that $\cH_i$ holds for some $i \geq 1$. On the one hand, we know from Lemma \ref{lem:struct2} that
\[
C(\tau_i) \eqlaw ER'(M_i,p_{\tau_i}).
\]
Now, how do we get $B(t_i)$? From the construction, $B(t_i)$  is constructed from the graphs $\cGN_{t_i}(N_i,k)$. Each has distribution $ER(N_i,p_{t_i})$, and we choose the $K(i)$-th graph, the first which has no large component. As in the previous section, it is thus distributed as an $ER'(N_i,p_{t_i})$ graph. By $\cH_i$, this has the same distribution as a $ER'(M_i,p_{\tau_i})$ graph, and thus
\[
B(t_i) \eqlaw ER'(N_i,p_{t_i}) \eqlaw ER'(M_i,p_{\tau_i}) \eqlaw C(\tau_i).
\]
Then, by Markov property, $C(\tau_i + t)_{t \geq 0}$ and $B(t_i+t)_{t \geq 0}$ evolve in the same way until the next gelation event, and $\cH_{i+1}$ follows.
\end{proof}

Hence, from now on, we will merely forget about $(B(t))$, and assume that $(C(t))$ is constructed as above. In particular, it is coupled with the process $(D(t))$ that we introduce now.

Recall that, as mentioned in the previous section, the goal of this construction is to couple our process with a more simple process made of independent variables, so as to be able to compare them. This more simple process $(D(t))$ will be defined as follow, where we write $\cGN(n) = \cGN(n,1)$.
\begin{description}
\item[\textbf{Step 0}] Let $N_0 = N$ and $\s_0 = 0$. Consider $\cGN(N_0)$, up to the time $\s_1 = \s(N_0)$ when a large component, of size $g_1 = g(N_0)$, appears. Define $D(t) = \cGN_t(N_0)$ for $t \in [0,\s_1)$, let $N_1 = N_0 - g_1 = s(N_0)$ and go to step 1.
\item[\textbf{Step i}] Consider here the graph processes $\cGN(N_i)$. Take $\s_{i+1} = \s(N_i) \vee \s_i$, $g_{i+1} = g(N_i)$ and define $D(t) = \cGN_t(N_i)$ for $t \in [\s_i,\s_{i+1})$, where this interval might be empty if $\s(N_i) \leq \s_i$. Then let $N_{i+1} = N_i - g_{i+1} = s(N_i)$ and go to step $i+1$.
\end{description}
In other words, we consider the same model as above, but taking always $K(i) = 1$, in the spirit of what we hinted at at the end of the previous section. It should be clear that $(D(t))$ is easier to study than $(C(t))$. Our next claim is that under \eqref{eq:alpha}, $D$ and $C$ are barely any different, at least on compact intervals. To this end, define
\[
I = \inf \{ i \geq 1, K(i) \neq 1 \}, \quad E(T) = \{t_I > T\},
\]
so that $E(T)$ is the event that $K(i) = 1$ for all the $\cGN$ we consider before time $T$. Hence, on $E(T)$, the processes $(C(t), t \in [0,T])$ and $(D(t), t \in [0,T])$ are equal. As usual, we will write $\EN(T)$ when we want to insist on the dependence on $N$. We will prove the following.

\begin{lemma} \label{lem:equivasymp}
If $(\a(N))$ is such that \eqref{eq:alpha} holds, then, for all $T \geq 0$, $\P(\EN(T)) \to 1$.
\end{lemma}

\begin{proof}
Take $\nu = \nu_T$ as in Lemma \ref{lem:posconc}. Let also $\k$ as in Proposition \ref{prop:atgel} (with $\dl = 1/2$ say, but it does not matter). Now note that, $K(i) = 1$ means that $\cGN_{t_i}(N_i,1)$ has no large component. But
\[
\cGN_{t_i}(N_i,1) \eqlaw ER(N_i,p_{t_i}) = ER(s(N_{i-1}),p_{\s(N_{i-1})}),
\]
and we know by Point 4 of Proposition \ref{prop:atgel} that this graph has no large component with probability greater than
\[
1 - N \exp \left ( - \k \frac{\a(N)^3}{N^2} \right ),
\]
provided $N_{i-1} \geq \nu N$. Then, since $t_i \leq T$ and $n_T \geq \nu$ implies $N_i \geq \nu N$, we have
\begin{align*}
1 - \P \left ( \EN (T) \right ) & = \P(t_I \leq T) \\
& \leq \P (t_I \leq T, n_T \geq \nu) + \P(n_T < \nu) \\
& = \P \left ( \bigcup_i \left \{ t_i \leq T, \; K(i) \neq 1, \; n_T \geq \nu \right \} \right ) + \P(n_T < \nu) \\
& \leq \P \left ( \bigcup_i \left \{ K(i) \neq 1, \; N_i \geq \nu N \right \} \right ) + \P(n_T < \nu) \\
& \leq N \times N \exp \left ( - \k \frac{\a(N)^3}{N^2} \right ) + \P(n_T < \nu),
\end{align*}
where we use at the last line that there are at most $N$ (even $N / \a(N)$) gelation events. The first part on the last RHS tends to 0 by \eqref{eq:alpha}, the second as well by choice of $\nu$ as in Lemma \ref{lem:posconc}.
\end{proof}

Recalling that $(C(t), t \in [0,T])$ and $(D(t), t \in [0,T])$ are equal on $E(T)$, we may thus state our conclusion as follows.

\begin{lemma} \label{lem:equiv}
If $(\a(N))$ is such that \eqref{eq:alpha} holds, then, for all $T > 0$,
\[
\P \left ( (D(t), t \in [0,T]) = (C(t), t \in [0,T]) \right ) \to 1.
\]
\end{lemma}

Hence, anything that happens w.h.p. for $(D(t))$ also happens w.h.p. for $(C(t))$, at least on compact intervals. We will thus only have to study $(D(t))$ from now on. Its main properties are summarized in the following section.

\subsection{Gelation times}

Recall that $(\s_i)$ is the sequence of (possibly equal) gelation times in the alternative model $(D(t))$.

\begin{lemma} \label{lem:geltimes}
Assume \eqref{eq:alpha}. Then for all $\nu> 0$ and $\dl > 0$, there exists $\k > 0$ such that, for all $i$ such that $N_i \geq \nu N$, with probability greater than
\[
1 - N \exp \left ( - \k \frac{\a(N)^3}{N^2} \right ),
\]
the following hold.
\begin{enumerate}
\item The gelation time is related to the mass in solution by
\[
\frac{N}{N_i} \left ( 1 - \left ( \frac12 + \dl \right ) \frac{\a(N)}{N_i} \right ) \leq \s_i \leq \frac{N}{N_i} \left ( 1 - \left ( \frac12 - \dl \right ) \frac{\a(N)}{N_i} \right )
\]
and
\[
\frac{N}{N_{i-1}} \left ( 1 + \left ( \frac12 - \dl \right ) \frac{\a(N)}{N_{i-1}} \right ) \leq \s_i \leq \frac{N}{N_{i-1}} \left ( 1 - \left ( \frac12 - \dl \right ) \frac{\a(N)}{N_{i-1}} \right )
\]
\item The time before the next gelation event verifies
\[
\frac{N}{N_i^2} \a(N) (1 - \dl)\leq \s_{i+1} - \s_i \leq \frac{N}{N_i^2} \a(N) (1 + \dl).
\]
\item The mass lost at a gelation event enjoys
\[
\a(N) \leq N_i - N_{i+1} \leq (1 + \dl) \a(N).
\]
\end{enumerate}
\end{lemma}

\begin{proof}
Let $\k$ be as in Proposition \ref{prop:atgel}, and let us use the same notation. All the events below are implied to happen with probability greater than $1 - N \exp (- \k \a(N)^3/N^2)$.

By definition, $\s_i = \sN(N_{i-1})$ and $N_i = \esN(N_{i-1})$. Point 1 and 3 of Proposition \ref{prop:atgel} then ensure that
\[
\sN_-(N_{i-1}) \leq \s_i \leq \sN_+(N_{i-1}), \quad \vsN_-(N_i) \leq \s_i \leq \sN_+(N_i)
\]
when $N_i \geq \nu N$. A simple Taylor expansion provides the first part of the result, for $N$ large enough, since $\a(N)/N \to 0$. From this, the second part is a straightforward computation, and the last inequality is just Point 2 of Proposition \ref{prop:atgel}. As usual, the constants can be absorbed by taking a smaller $\k$.
\end{proof}

This results concerns the gelation times in the alternative model. This allows to study the gelation times in our original model, as summarized below.

\begin{prop} \label{prop:geltimes}
Assume \eqref{eq:alpha}. Then for any $T > 0$ and $\dl > 0$, with probability tending to one, it holds that
\[
\frac{1}{n_{\tau_i}^2} \frac{\a(N)}{N} (1 - \dl)\leq \tau_{i+1} - \tau_i \leq \frac{1}{n_{\tau_i}^2} \frac{\a(N)}{N} (1 + \dl)
\]
and
\[
\frac{\a(N)}{N} \leq n_{\tau_i} - n_{\tau_{i+1}} \leq (1 + \dl) \frac{\a(N)}{N}
\]
for all $i$ with $\tau_i \leq T$.
\end{prop}

\begin{proof}
First, note that Points 2 and 3 in Lemma \ref{lem:geltimes} hold for all $i$ with $N_i \geq \nu N$ with probability at least
\[
1 - N \times N \exp \left ( - \k \frac{\a(N)^3}{N^2} \right )
\]
since there are at most $N$ gelation events. Since $(\a(N))$ verifies \eqref{eq:alpha}, this probability tends to 1.

Now, denote by $d_t = \# D(t) / N$ the concentration in the alternative model. By Lemma \ref{lem:equiv}, w.h.p., $\tau_i = \s_i$ and $n_{\tau_i} = d_{\s_i}$ for all $i \leq T$. Finally, by construction, if $K(i) = 1$, then $N_i  = d_{\s_i} N$. In particular, this holds on the event of high probability $\EN(T)$, as long as $\tau_i \leq T$. The result is then just replacing $N_i$ by $n_{\tau_i} N$ in Points 2 and 3 of Lemma \ref{lem:geltimes}.
\end{proof}

\section{Convergence} \label{sec:conv}

In this section we finish the proof of Theorem \ref{th:conv}. We will classically proceed in two steps: first prove the tightness, and then prove that the limit is uniquely given by the result.

\begin{lemma}
Assume \eqref{eq:alpha}. Then the sequence $(\nN)$ is tight in $D([0,\pinf))$, and any limiting point is Lipschitz-continuous.
\end{lemma}

\begin{proof}
Fix $T > 0$ and define
\[
w^N(\dl) = \sup_{s,t \leq T, \; |s-t| \leq \dl} \left | \nN_s - \nN_t \right |.
\]
According to Theorem 16.8 and (12.7) in \cite{BillingsleyCPM}, to check the tightness, it suffices to show that
\[
\lim_{a \to \pinf} \limsup_{N \to \pinf} \P \left ( \sup_{t \in [0,T]} \left | \nN_t \right | > a \right ) = 0
\]
and for all $\eps > 0$,
\[
\lim_{\dl \to 0} \limsup_{N \to \pinf} \P \left ( w^{(N)}(\dl) > \eps \right ) = 0.
\]
The first point is obvious since $(\nN)$ is bounded by 1. For the second point, taking $\dl = 1/2$ in Proposition \ref{prop:geltimes}, implies that w.h.p., all the gelation times $(\tau_i)$ before $T$ are more than $\a(N) / 2 N$ apart. But a gelation event makes at most $2 \a(N)$ particles fall into the gel, so that the concentration lost on an interval of size $\dl$ in $[0,T]$ is at most
\[
\frac1N \frac{2 N}{\a(N)} \dl 2 \a(N) =  4 \dl.
\]
There is merely a fencepost error here, and we thus have
\[
w^{(N)}(\dl) \leq 4 \dl + \frac{2 \a(N)}{N}
\]
w.h.p., and the tightness, as well as the Lipschitz-continuity, follow.
\end{proof}

Let us now finish the proof of Theorem \ref{th:conv}. We can assume, by the previous result, that $(\nN)$ converges to some continuous $m$ in $\D([0,\pinf))$. We shall prove that there is a unique limit point given by \eqref{eq:nt}, which suffices to conclude.

\begin{lemma} 
Almost surely, any limit point of $(\nN)$ verifies \eqref{eq:nt}.
\end{lemma}

\begin{proof}
We already know, say by \eqref{eq:Tgel0}, that no component has fallen into the gel at a time $t < 1$ w.h.p., so that $m_t = 1$ for $t < 1$.

Now, consider $t > 1$, $s > 0$. Fix $\dl > 0$. All the following events will happen w.h.p. At least one gelation event has happened at time $t$ and  moreover, by Proposition \ref{prop:geltimes} with $T = t+s$, if $t \leq \tau_i \leq t+s$, then
\[
\frac{\a(N)}{N(\nN_t)^2} (1 - \dl) \leq \frac{\a(N)}{N(\nN_{\tau_i})^2} (1 - \dl) \leq \tau_{i+1} - \tau_i \leq \frac{\a(N)}{N(\nN_{\tau_i})^2} (1 + \dl) \leq \frac{\a(N)}{N(\nN_{t+s})^2} (1 + \dl).
\]
Consequently, the number of gelation events between $t$ and $t+s$ is at least
\[
\frac{N}{\a(N)} \frac{1}{1+\dl} (\nN_{t+s})^2 s
\]
and at most
\[
1 + \frac{N}{\a(N)} \frac{1}{1-\dl} (\nN_t)^2 s.
\]
Still by Proposition \ref{prop:geltimes}, a gelation events makes at most $(1+\dl) \a(N)$ particles fall in the gel. Hence, the quantity that falls in the gel between $t$ and $t + s$ is at least
\[
N \frac{1}{1+\dl} (\nN_{t+s})^2 s
\]
and at most
\[
(1 + \dl) \a(N) + N \frac{1+\dl}{1-\dl} (\nN_t)^2 s.
\]
Passing to the limit shows that, for all $\dl$, almost surely,
\[
\frac{1}{1+\dl} m_{t+s}^2 \times s \leq m_t - m_{t+s} \leq \frac{1+\dl}{1-\dl} m_t^2 \times s.
\]
By continuity, this extends to all $t \geq 1$ and $s \geq 0$, and we may let $\dl \to 0$, to obtain
\[
m_{t+s}^2 \times s \leq m_t - m_{t+s} \leq m_t^2 \times s
\]
for all $t \geq 1$, $s \geq 0$, so that
\[
\dfdt m_t = - m_t^2, \quad t \geq 1.
\]
Since $m_t = 1$ for $t < 1$, by continuity $m_1 = 1$, and the result follows.
\end{proof}

\section{Conclusion} \label{sec:conclusion}

\subsection{Microscopic description: proof of the last results}

Let $K_n$ a complete graph on $n$ vertices, whose edges are given i.i.d mean $n$ exponential lengths.  
In \cite{AldousSteele} it is proven (see Theorem 4.1 and (4.11) of that reference) that one has weak local convergence of $K_n$ towards the Poisson Weighted Infinite Tree (PWIT) where the edge lengths of children of distinct vertices are independent realizations of Poisson point processes with intensity one. We refer to \cite{AldousSteele}, Section 2 of for a precise definition of weak local convergence\footnote{In this definition of \cite{AldousSteele}, graphs are seen as metric spaces, and edge lengths play the obvious role in defining the distance. For graphs whose edges lengths have not been precised --- as for the graphs of this article --- we adopt the usual convention that they are identically equal to 1.}, and Section 4.2 of \cite{AldousSteele} for a precise definition of the PWIT. 

Letting $\widetilde{p_n}=-\ln(1-p_n)$, it is straightforward that if $\widetilde{K_n}$ is obtained by performing an edge percolation on $K_n$ in keeping only edges whose length lies below $n \widetilde{p_n}$, then the graph structure of $\widetilde{K_n}$ is exactly that of a $ER(n,p_n)$ random graph. Moreover, observe that when $n p_n \to c$, $n\widetilde{p_n} \to c$ as well. Thus, when suppressing all edges in the PWIT whose length is above $c$, we end up with a graph structure which is exactly that of a Galton-Watson tree with $\cP(c)$ reproduction law.

To conclude\footnote{This weak local convergence result for ER random graphs may not be stated directly in \cite{AldousSteele}, but the result was in fact well-known since earlier works (see e.g. \cite{Aldous90}, or Lemma 2.2 of \cite{Aldous98TVM})}, under the assumption $np_n \to c$, the result of Aldous and Steele implies the weak local convergence of a typical component of $ER(n,p_n)$ towards a Galton-Watson tree with offspring distribution $\cP(c)$. 

Finally recall that, whenever $c \le 1$, a Galton-Watson tree with offspring distribution $\cP(c)$ is almost surely finite. Thus, when $n p_n \to c \le 1$, weak local convergence of a typical component of $ER(n,p_n)$ (denoted $\cC_{ER}(n,p_n)$) precisely says that for any finite rooted tree $\cT$,
\[
\P(\cC_{ER}(n,p_n) = \cT) \to \P(\mathrm{GW}(\cP(c)) = \cT).
\]

From the construction of Section~\ref{sec:altmodel} and Lemma~\ref{lem:equiv}, for any fixed $t \ge 0$, w.h.p, $C(t)=D(t)$ and $D(t)$ is exactly a $ER(N\nN_t, \pN_t)$. Of course, by Theorem~\ref{th:conv}, $N\nN_t \pN_t \to t n_t \le 1$, which concludes the proof of Theorem \ref{th:typical}.

\subsection{Recovering the macroscopic results} \label{sec:macro}

When the initial conditions to Smoluchowski's equation are assumed monodisperse, i.e. given by $c_0(m) = \unn{m=1}$, there is an earlier version of well-posedness obtained in \cite{Kokholm}, and the solution in that case is truly explicit. More precisely, let us define
\[
B(\l,m) = \frac{(\l m)^{m-1}}{m!} e^{- \l m}
\]
to be the Borel distribution with parameter $\l$. It is the distribution of the total progeny of a Galton-Watson process with reproduction law $\cP(\l)$ (see e.g. Dwass' formula \cite{Dwass}). 

According to \cite{Kokholm,NZ}, the solution to \eqref{eq:smolu} with monodisperse initial conditions is given by 
\begin{equation} \label{eq:concmono}
c_t(m) =
\begin{cases}
\frac1m B(t,m) & t \leq 1, \\
\frac{1}{m t} B(1,m) & t \geq 1.
\end{cases}
\end{equation}
Theorem \ref{th:typical} provides a better understanding of this explicit solution. Indeed, the distribution of the size of the cluster to which a particle chosen uniformly at random belongs is size-biased, and thus
\[
c\eN_t(m) = n\eN_t \frac1m \P(|\cC(t)| = m). 
\]
Consequently, for $t \geq 1$,
\[
c_t(m) = \lim_{N \to \infty} \cN_t(m) = \frac1m \lim_{N \to \infty} n\eN_t \P(|\cC(t)| = m) = \frac1m \frac1t P(|\mathrm{GW}(\cP(1))| = m)  = \frac{1}{mt} B(1,m)
\]
Clearly, the same reasoning would also work for $t < 1$, only instead the Galton-Watson tree would have $\cP(t)$ offspring distribution. 
 
Observe further, as we mentioned at the end of the introduction, that the ``macroscopic'' SOC in the sense of \cite{RathToth} follows in fact directly from \eqref{eq:concmono}. Indeed, for $t \geq 1$, we have as $k \to \pinf$
\[
\sum_{m \geq k} m c_t(m) = \frac{1}{t} \sum_{m \geq k} B(1,m) \sim \frac{1}{t \sqrt{2\pi}} \sum_{m \ge k} m^{-3/2} \sim  \frac{\sqrt{2}}{t \sqrt{\pi}} k^{-1/2}. 
\]
where we used Stirling's formula to get the first equivalent.  

Here again, Theorem~\ref{th:typical} provides a better understanding: the quantity $\sum_{m \geq k} B(1,m)$ above   
is the probability for a critical $\cP(1)$ Galton-Watson tree to have size greater than $k$, which is well-known to behave in $k^{-1/2}$, as compared to an exponential decay in the strictly subcritical case, corresponding to $t < 1$.

\subsection{Aldous's conjecture} \label{sec:Aldousconj}

 Let us now discuss other features of Aldous' Conjecture 3.6 of \cite{Aldous98TVM}.

\paragraph{Clusters in solution} Recall from Section \ref{sec:introAldous} that $\cC(t)$ is the cluster of the particle 1 at time $t$, and $\ZN$ the (a.s. finite) time when it falls into the gel. As a consequence of the main assessment of the conjecture that we proved in Section \ref{sec:introAldous}, Aldous deduces in particular that for any $t \ge 1$, 
\begin{equation} \label{eq:ParetoBorel}
\P(|\cC(t)| = k) =  \P(|\cC(t)| = k, \ZN > t ) \underset{N \to \infty}{\longrightarrow} \frac{1}{t} \frac{k^{k-1}e^{-k}}{k!}.
\end{equation}
In fact, the above also follows directly from our results. We already checked that $\ZN \to Z \sim$ Pareto$(1,1)$. Furthermore, knowing that $1$ is in solution at a $t \ge 1$, the distribution of the size of its cluster converges to that of the size of a $GW(\cP(1))$-tree, i.e. a Borel$(1)$ distribution. 

At this point we should observe that \eqref{eq:ParetoBorel} already came as an immediate consequence of \cite{FournierMLP}, and similarly of \cite{Rath} for a closely related model (see the next paragraph for more details). These two works establish convergence of concentrations of clusters of a given size towards the solution to Smoluchowski's equation. In the case of monodisperse initial conditions, the explicit solution \eqref{eq:concmono} is but a rephrasing of \eqref{eq:ParetoBorel}. Thus, this part of Aldous' conjecture \eqref{eq:ParetoBorel} was in fact already established, and the proof is valid under the weaker assumption \eqref{eq:alpha0}.

Nevertheless, the convergence of the law of $\ZN$ towards a Pareto$(1,1)$ is much less evident from the approach in \cite{FournierMLP,Rath}  as it would also require proving that :  
\[
\P(\ZN >t) = \sum_{k \ge 0} k \cN_t(k) \underset{N \to \infty}{\longrightarrow} \sum_{k \ge 0}
k c_t(k) = \frac{1}{t}.
\]
 Note that this is not a direct consequence of their results, or of Theorem~\ref{th:convconc}, since the convergence of cluster concentrations is only in $\ell^1$. As we just explained, our Theorem~\ref{th:conv} (and its rather lengthy proof) does imply the above, but we are left to wonder if it is possible to show the convergence of the distribution of $\ZN$ in a more direct way.  

\paragraph{Clusters in the gel} It is reasonable to further ask what happens to particles in the gel\footnote{Of course, the way we decided to deal with the gel is somewhat arbitrary. If any links were to be allowed between particles in the gel, this would not change anything to our results, but this would of course cause the emergence of a single giant. A more physically relevant picture would be to introduce some spatial structure for gel components, in which case the problem would become much harder to deal with.}.
In \cite{AldousFP}, Aldous claims that the distribution of the component in which $1$ ultimately ends up, i.e. $\cC(Z)$, converges, in the local weak sense, towards $\cG^0(1/V)$, which Aldous shows to be a $GW(\cP(1))$ conditioned on survival. This local weak convergence follows from (\ref{eq:aldousconj}), which we proved in the introduction.  

In fact, we can be a bit more precise: this component is the emerging giant in a slightly supercritical ER graph.
 As should be clear from the alternative model, the gel is w.h.p a collection of clusters of size $\a(N)+o(\a(N))$.  Moreover, by Lemma~\ref{lem:equiv}, w.h.p, all components which fall before a given $T$ are emerging giants in some ER dynamic graph. With arbitrarily large probability, for a large enough $T$, $Z \leq T$ and thus $\cC(Z)$, the component in which $1$ ultimately ends up, is one of these emerging giants. Of course w.h.p such emerging giant is not a tree (the number of surplus edges typically diverges), but cycles are typically very large so it does not contradict the weak local convergence result. For this reason, a stronger type of convergence would fail.  

This brings us to the last part of the conjecture in \cite{AldousFP}, namely that $Z$ and $\cC(Z)$ should be asymptotically independent. This turns out to be false. First, the concentration of clusters in solution at time $\ZN$ is $\nN_{\ZN}$ whose limit in law is $1/Z$.  By Proposition~\ref{prop:atgel}, $\cC(\ZN)$ is w.h.p, the emerging giant of size $\a(N)+o(\a(N))$ in a
\[
ER \left ( N \nN_Z, Z \left ( 1 + \frac12 \frac{Z \a(N)}{N} + o(\a(N)/N)\right ) \right )
\]
graph. Letting $n=N \nN_Z$, $\varepsilon_n = \a(N)/2n$, the result of Proposition~\ref{prop:boundS} allows to check that at step $k=[tn \varepsilon_n]$ of the exploration process of a $ER(n,p_n)$ graph, the expected number of (ignored) surplus edges is approximately $ \varepsilon_n^2 (t-t^2/2)$. It is then straightforward to establish that    
\[
\frac{N^2}{\a(N)^3} \mathbb{E} \left[ \#\{ \mbox{surplus edges in } \cC(Z)\}  \middle | Z \right] \underset{N \to \pinf}{\longrightarrow} \frac{Z^2}{12},
\] 
so in particular,  $\cC(Z)$ and $Z$ can not be asymptotically independent. 

We believe that, for at least hoping to obtain such asymptotic independence, one would need to consider instead a decreasing threshold $\a(N_t)$.

\subsection{Assumption on the threshold}

One may rightfully question the assumption \eqref{eq:alpha} on $(\a(N))$. With the techniques of this paper, it could be slightly relaxed. Going through the proofs shows that we only need
\[
\frac{N^2}{\a(N)} \exp \left ( - \k \frac{\a(N)^3}{N^2} \right ) \to 0,
\]
for any $\k > 0$, what can be achieved with logarithmic factors. 

But our techniques can not be extended to the case of smaller thresholds. The main reason is that, even though Lemma~\ref{lem:struct2} still holds, the conditioning becomes asymptotically non trivial. Think for instance of the case $\a(N)= cN^{2/3}$ which brings the model at its first gelation time right inside the critical window, as studied by Aldous \cite{AldousCriticalRG}. In the limit $N \to \infty$ the conditioning would corresponds to conditioning a parabolically drifted BM to have no excursion above its infimum of size greater than $c$, an event of probability $p_c \in (0,1)$. In fact, in the case $\a(N) \ll N^{2/3}$ the conditioning could even be degenerate. 

Hence, for smaller thresholds, the configuration in solution at some time past gelation is far from obvious, and it may strongly affect at least the distribution of the largest clusters in solution. 

We {\em conjecture} however that both Theorem~\ref{th:conv} and \ref{th:typical} remain true under \eqref{eq:alpha0}. 
The intuition is that only the largest clusters are affected by the asymptotically non-trivial conditioning. Since those largest clusters asymptotically account for a vanishing proportion of vertices, and since they live very briefly in solution, we believe that the typical microscopic evolution remains asymptotically the same.

\bigskip

{\bf Aknowledgement} :  
The authors wish to thank Bal\'asz R\'ath for pointing us to \cite{AldousFP} and \cite{Rath}, and also for making several interesting remarks concerning a first version of this paper. 

\bibliographystyle{abbrv}
\bibliography{bibli}

\end{document}